\documentclass[aos,preprint]{imsart}

 \usepackage{amssymb,amsfonts,graphicx,amsthm,nicefrac,mathtools,bm}
 \usepackage{hyperref}
  \RequirePackage{natbib}
 \usepackage[dvipsnames]{xcolor}
 \usepackage{graphics}
 \usepackage{layouts}

\usepackage[]{caption}
 \usepackage{tikz}
 \usetikzlibrary{fit}
 \input{tikzlibrarybayesnet.code}
 \usetikzlibrary{positioning,shadows,arrows}
 \usetikzlibrary{shapes.geometric}
 \usetikzlibrary{patterns}
 \tikzset{
 	font={\fontsize{11pt}{12}\selectfont}}
 
 \arxiv{arXiv:1803.06790}
 
 \newtheorem{theorem}{Theorem}
 
 \newtheorem{lemma}{Lemma}
 \newtheorem{corollary}{Corollary}

 \theoremstyle{definition}
 
 \theoremstyle{remark}
 \newtheorem{remark}{Remark}
 
 \newcommand{\PP}[1]{\textnormal{Pr}\!\left\{{#1}\right\}} 
 \newcommand{\EE}[1]{\mathbb{E}\left[{#1}\right]} 
 \newcommand{\Ep}[2]{\mathbb{E}_{#1}\left[{#2}\right]}
 \newcommand{\EEst}[2]{\mathbb{E}\left[{#1}\ \middle| \ {#2}\right]} 
 \newcommand{\PPst}[2]{\text{Pr}\!\left\{{#1}\ \middle| \ {#2}\right\}} 
 \def\R{\mathbb{R}}

 \def\independenT#1#2{\mathrel{\rlap{$#1#2$}\mkern2mu{#1#2}}}
 \newcommand\independent{\protect\mathpalette{\protect\independenT}{\perp}}

 \newcommand{\ignore}[1]{}

 \newcommand{\fdp}{\textnormal{FDP}}
 \newcommand{\fdr}{\textnormal{FDR}}
 \newcommand{\fdx}{\textnormal{FDX}}
 \newcommand{\fdphat}{\widehat{\textnormal{FDP}}}
 \newcommand{\fdpbar}{\overline{\textnormal{FDP}}}

 \newcommand{\vh}{\widehat{V}}

 \def\F{\mathcal{F}}
 \def\G{\mathcal{G}}
 
 \def\H{\mathcal{H}}
 \def\cR{\mathcal{R}}

 \newcommand{\sort}{\textnormal{sort}}
 
 \newcommand{\preordersel}{\textnormal{preorder-sel}}
 
 \newcommand{\online}{\textnormal{online}}

 \newcommand{\simpleol}{\textnormal{online-simple}}
 \newcommand{\adaptiveol}{\textnormal{online-adaptive}}
 
\begin{document}

\begin{frontmatter}
	\title{{Simultaneous high-probability bounds\\ on the false discovery proportion in\\  structured, regression, and online settings}}
	\runtitle{Simultaneous high-probability FDP bounds}
	
	\begin{aug}
		\author{\fnms{Eugene} \snm{Katsevich}\thanksref{t1,m1}\ead[label=e1]{ekatsevi@stanford.edu}} \and
		\author{\fnms{Aaditya} \snm{Ramdas}\thanksref{m2}\ead[label=e2]{aramdas@cmu.edu}}
		
		\thankstext{t1}{E.K.'s work was supported by the Fannie and John Hertz Foundation.}
		\runauthor{E. Katsevich and A. Ramdas}
		
		\affiliation{Stanford University\thanksmark{m1} and Carnegie Mellon University\thanksmark{m2}}
		
		\address{
Eugene Katsevich\\
		Department of Statistics\\
			Stanford University\\
			\printead{e1}\\
			\phantom{E-mail:\ }}
		
		\address{
Aaditya Ramdas\\
		Department of Statistics and Data Science\\
			Carnegie Mellon University \\
			\printead{e2}\\
			\phantom{E-mail:\ }}
		
	\end{aug}
	
\begin{abstract}


While traditional multiple testing procedures prohibit adaptive analysis choices made by users, \cite{GS11mt} proposed a simultaneous inference framework that allows users such flexibility while preserving high-probability bounds on the false discovery proportion (FDP) of the chosen set. In this paper, we propose a new class of such simultaneous FDP bounds, tailored for nested sequences of rejection sets. While most existing simultaneous FDP bounds are based on closed testing using global null tests based on sorted p-values, we additionally consider the setting where side information can be leveraged to boost power, the variable selection setting where knockoff statistics can be used to order variables, and the online setting where decisions about rejections must be made as data arrives. Our finite-sample, closed form bounds are based on repurposing the FDP estimates from false discovery rate (FDR) controlling procedures designed for each of the above settings. These results establish a novel connection between the parallel literatures of simultaneous FDP bounds and FDR control methods, and use proof techniques employing martingales and filtrations that are new to both these literatures. We demonstrate the utility of our results by augmenting a recent knockoffs analysis of the UK Biobank dataset.

\end{abstract}
	
	\begin{keyword}
		\kwd{false discovery rate}
		\kwd{FDR}
		\kwd{multiple testing}
		\kwd{post hoc confidence bounds}
		\kwd{false discovery exceedance}
		\kwd{FDX}
		\kwd{uniform martingale concentration}
	\end{keyword}
	
\end{frontmatter}

\section{Introduction}

%

\subsection{Multiple testing and exploration}

Consider testing a set of hypotheses $\mathcal H = \{H_1, \dots, H_n\}$, which we identify with $[n] \equiv \{1, \dots, n\}$. The false discovery proportion (FDP) of a rejection set $\cR \subseteq [n]$ is defined
\begin{equation*}
\fdp(\cR) \equiv \frac{|\cR \cap \H_0|}{|\cR|} \equiv \frac{V}{R},
\end{equation*}
where $\H_0 \subseteq \H$ is the set of nulls and $\fdp(\cR) \equiv 0$ when $\cR = \varnothing$ by convention (we use the $\equiv$ symbol for definitions). Based on the data at hand, a multiple testing procedure returns a rejection set $\cR^*$, and the Type-I error of such a procedure can be evaluated via different properties of the FDP distribution. For example, the false discovery rate (FDR) is defined as the mean of the FDP \citep{BH95} and the false discovery exceedance (FDX) is defined as the probability that FDP exceeds a pre-chosen threshold $\gamma$ \citep{lehmann2005generalizations}:
\begin{equation*}
\fdr \equiv \EE{\fdp(\cR^*)} \quad \text{and} \quad \fdx \equiv \PP{\fdp(\cR^*) > \gamma}.
\end{equation*}
A multiple testing procedure is said to control an error rate if it falls below a pre-specified level; e.g., a procedure controls the FDR at level $q$ if $\fdr \leq q$.

This multiple testing paradigm has been very successful as the workhorse of scientific discovery for the past several decades. However, \cite{GS11mt} (GS) argued that this prevailing paradigm may not be flexible enough to accommodate for the exploratory nature of modern large-scale data analysis: target levels for the FDP like $q$ or $\gamma$ are chosen in advance, and rejection sets obtained from FDR- or FDX-controlling procedures cannot be grown or shrunk without invalidating their guarantees. Therefore, the paradigm leaves little room for scientists seeking to use their domain expertise to adaptively select a rejection set while maintaining valid inferential guarantees. They may attempt to do so by applying a multiple testing procedure with different nominal levels and choosing one of the resulting rejection sets post hoc. They may also exclude rejected hypotheses that do not align with their scientific priors, or include unrejected hypotheses that were close to the rejection threshold. However, these practices may lead to an excess of false positives. 

Motivated by these considerations, GS proposed a complementary \textit{simultaneous inference} paradigm. In this paradigm, one constructs FDP upper bounds $\fdpbar(\cR)$ that hold uniformly across all sets $\cR$ with high-probability:
\begin{equation}
\PP{\fdp(\cR) \leq \fdpbar(\cR) \text{ for all } \cR \subseteq \H} \geq 1-\alpha.
\label{simultaneous}
\end{equation}
Such bounds allow the scientist to inspect any pairs $(\cR, \fdpbar(\cR))$ and freely choose the rejection set $\cR^*$ whose content and FDP bound suits them. Given the simultaneous nature of statement \eqref{simultaneous}, the upper bound on FDP continues to hold on the chosen set despite the user's data-dependent decision:
\begin{equation*}
\begin{split}
&\PP{\fdp(\cR^*) \leq \fdpbar(\cR^*)} \geq \\
&\PP{\fdp(\cR) \leq \fdpbar(\cR) \text{ for all } \cR \subseteq \H} \geq 1-\alpha.
\end{split}
\label{spotting_validity}
\end{equation*}
GS obtain such bounds by building on the closed testing principle, where a \textit{local test} $\phi_{\cR}$ (i.e. a test of the global null for a restricted set of hypotheses) is performed for each subset of hypotheses $\cR \subseteq \H$. The results of all these local tests are aggregated to form a bound $\fdpbar$ that provably satisfies \eqref{simultaneous}. Since then, there has been much exciting work on new algorithms and computational shortcuts for simultaneous FDP control, mostly based on closed testing, but these have been somewhat disconnected from the parallel growth in the FDR literature.

\subsection{A new class of simultaneous FDP bounds} \label{sec:new_class}
In this paper, we show that a variety of FDR procedures can be repurposed to obtain simultaneous FDP bounds, establishing a novel connection between FDR control and simultaneous FDP control. In particular, note that many FDR algorithms implicitly construct a path, or nested sequence of $n$ potential rejection sets
\begin{equation*}
\Pi\equiv(\cR_0,\dots,\cR_n), \text{ such that } \varnothing \equiv \cR_0 \subseteq \cR_1 \subseteq \cdots \subseteq \cR_n \subseteq [n].
\label{path}
\end{equation*}
Then, an estimate of the FDP 
\begin{equation}
\fdphat(\cR_k) \equiv \frac{a_0 + \vh(\cR_k)}{|\cR_k|}
\label{fdphat}
\end{equation}
is constructed for each $\cR_k \in \Pi$, where $\vh(\cR_k)$ is an estimate of $V(\cR_k) \equiv |\cR_k \cap \mathcal H_0|$ and $a_0 \geq 0$ is an additive regularization constant. This estimate is then used to obtain a cutoff point
\begin{equation}
k^* \equiv \max\{k: \fdphat(\cR_k) \leq q\},
\label{k_star}
\end{equation}
based on which the rejection set $\cR^* \equiv \cR_{k^*}$ is defined. 

Repurposing the path $\Pi$ and the estimate $\vh$, we propose the bound
\begin{equation}
\fdpbar(\cR_k) \equiv \frac{\overline V(\cR_k)}{|\cR_k|}; \quad \overline V(\cR_k) =\lfloor c(\alpha) \cdot (a + \vh(\cR_k))\rfloor,
\label{fdpbar}
\end{equation}
where $c(\alpha)$ are tight, explicit, dimension-independent constants such that 
\begin{equation}
\begin{split}
\label{simultaneous_selective}
\PP{\fdp(\cR) \leq \fdpbar(\cR) \text{ for all } \cR \in \Pi} \geq 1-\alpha,
\end{split}
\end{equation}
as long as the p-values satisfy an independence assumption. The constant $c(\alpha)$ depends implicitly on the regularization $a > 0$, which does not need to be the same as the original regularization $a_0$. Usually, we set $a = 1$. If desired, $\fdpbar$ can also be extended to all sets $\cR \subseteq \H$ to obtain a bound of the form~\eqref{simultaneous} through the process of \textit{interpolation} \citep{BetR17, GHS19}. Figure~\ref{fig:schematic} summarizes the proposed FDP bounds and how they compare and contrast to FDR methods and GS's simultaneous inference paradigm based on closed testing.

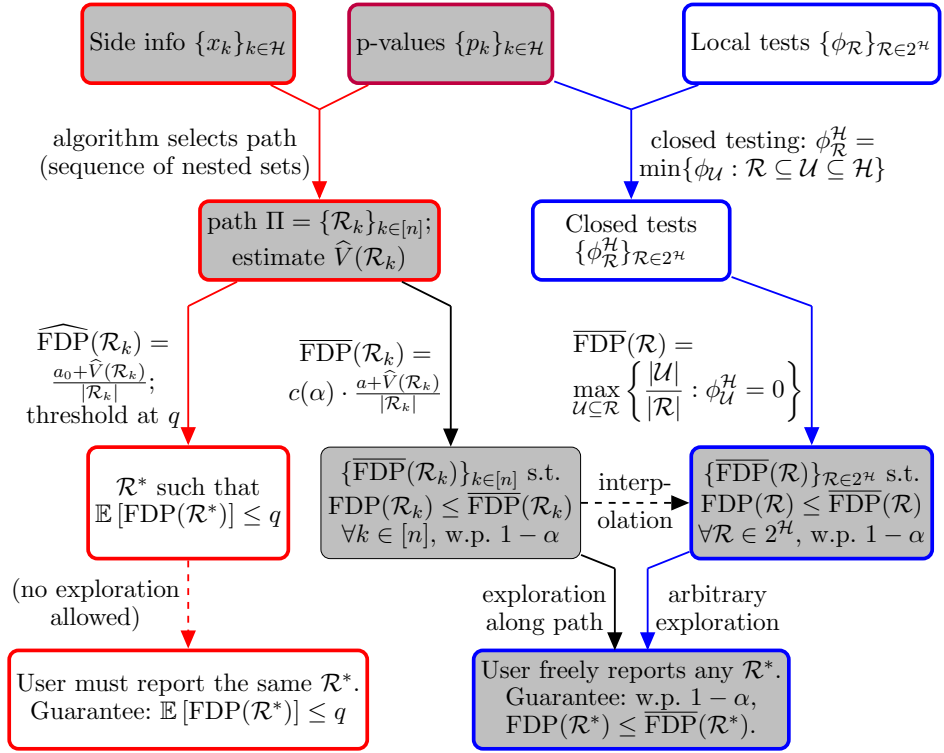
\begin{figure}[h!]
\resizebox{\textwidth}{!}{
\begin{tikzpicture}[]

\node[text opacity=1,fill=white, draw = red, ultra thick, align = center, minimum height = 1.2cm, rounded corners, fill = lightgray]  (sideinfo)  at(0,0)   {Side info $\{x_k\}_{k \in \H}$} ; 

\node[text opacity=1,fill=white, draw = purple, ultra thick, align = center, minimum height = 1.2cm, rounded corners, fill = lightgray]  (pvalues)  at(4,0)   {p-values $\{p_k\}_{k \in \H}$} ; 

\node[text opacity=1,fill=white, draw = blue, ultra thick, align = center, minimum height = 1.2cm,, rounded corners]  (localtests)  at(9.5,0)   {Local tests $\{\phi_{\cR}\}_{\cR \in 2^\H}$} ; 

\coordinate  (prepath)  at (2,-1)   {} ; 		
\coordinate  (preclosed)  at (6.75,-1)   {} ; 		

\node[text opacity=1,fill=white, draw = red, ultra thick, align = center, minimum height = 1.2cm, rounded corners, minimum width = 3cm, fill = lightgray]  (path)  at(2,-3)   {path $\Pi = \{\cR_k\}_{k \in [n]}$;\\ estimate $\vh(\cR_k)$} ; 

\node[text opacity=1,fill=white, draw = blue, ultra thick, align = center, minimum height = 1.2cm, rounded corners, minimum width = 3cm]  (closedtests)  at(6.75,-3)   {Closed tests \\ $\{\phi^{\H}_{\cR}\}_{\cR \in 2^\H}$};

\coordinate  (preFDRoutput)  at (0,-4)   {} ; 		
\coordinate  (prespottingoutput)  at (4,-4)   {} ; 		
\coordinate  (preGSoutput)  at (9.5,-4)   {} ; 	

\node[text opacity=1,fill=white, draw = red, ultra thick, align = center, minimum height = 1.7cm, rounded corners]  (FDRoutput)  at(0,-7)   {$\cR^*$ such that \\ $\EE{\fdp(\cR^*)} \leq q$} ; 

\node[text opacity=1,fill=white, draw = black, align = center, minimum height = 1.7cm, rounded corners, fill = lightgray]  (spottingoutput)  at(4,-7)   {$\{\fdpbar(\cR_k)\}_{k \in [n]}$ s.t. \\ $\fdp(\cR_k) \leq \fdpbar(\cR_k)$ \\ $\forall k \in [n]$, w.p. $1-\alpha$}; 

\node[text opacity=1,fill=white, draw = blue, ultra thick, align = center, minimum height = 1.7cm, rounded corners, fill=lightgray]  (GSoutput)  
at(9.5,-7)   {$\{\fdpbar(\cR)\}_{\cR \in 2^\H}$ s.t.\\ $\fdp(\cR) \leq \fdpbar(\cR)$ \\  $\forall \cR \in 2^\H$, w.p. $1-\alpha$}; 

\coordinate  (postspottingoutput)  at (6.5,-8)   {} ; 	
\coordinate  (postGSoutput)  at (7,-8)   {} ; 	

\coordinate  (FDPfinaloutputleft)  at (6.5,-9.25)   {} ; 	
\coordinate  (FDPfinaloutputright)  at (7,-9.25)   {} ; 

\node[text opacity=1,fill=white, draw = red, ultra thick, align = center, minimum height = 1.5cm, minimum width = 4.5cm, rounded corners]  (FDRfinaloutput)  at(0,-10)   {User  must report the same $\cR^*$. \\ Guarantee: $\EE{\fdp(\cR^*)} \leq q$} ; 

\node[text opacity=1,fill=white, draw = blue, ultra thick, align = center, minimum height = 1.5cm, minimum width = 4.5cm, rounded corners, fill = lightgray]  (FDPfinaloutput)  at(6.75,-10)   {User freely reports any $\cR^*$.\\ Guarantee: w.p. $1-\alpha$, \\ $\fdp(\cR^*) \leq \fdpbar(\cR^*)$.};


\draw[-, red, thick](sideinfo)--(prepath);
\draw[-, red, thick](pvalues)--(prepath);	
\draw[->, red, thick, text = black](prepath)--(path) node[midway,left, align=center] {algorithm selects path \\ (sequence of nested sets)};	

\draw[-, blue, thick](pvalues)--(preclosed);
\draw[-, blue, thick](localtests)--(preclosed);	
\draw[->, blue, thick, text = black](preclosed)--(closedtests) node[midway,right, align = center] {closed testing: $\phi_\cR^{\mathcal H} = $ \\ $\min\{\phi_{\mathcal U}: \mathcal R \subseteq \mathcal U \subseteq \mathcal H\}$};	

\draw[-, red, thick](path)--(preFDRoutput);
\draw[->, red, thick, text = black](preFDRoutput)--(FDRoutput) node[midway,left, align = center] {$\fdphat(\cR_k) =$ \\ $\frac{a_0 + \vh(\cR_k)}{|\cR_k|}$; \\ threshold at $q$};

\draw[-, thick](path)--(prespottingoutput);
\draw[->, thick](prespottingoutput)--(spottingoutput) node[midway,left, align = center] {$\fdpbar(\cR_k) =$ \\ $c(\alpha) \cdot\frac{a + \vh(\cR_k)}{|\cR_k|}$};

\draw[-, blue, thick](closedtests)--(preGSoutput);
\draw[->, blue, thick, text = black](preGSoutput)--(GSoutput) node[midway,left, align = left] {$\fdpbar(\cR) = $ \\ $\displaystyle \max_{\mathcal U \subseteq \mathcal R} \left\{\frac{|\mathcal U|}{|\mathcal R|}: \phi_{\mathcal U}^{\mathcal H} = 0\right\}$};

\draw[dashed,->, thick](spottingoutput)--(GSoutput) node[above, midway, align = center] {interp-} node[below, midway, align = center] {olation};

\draw[dashed,->, red, text = black, thick](FDRoutput)--(FDRfinaloutput) node[midway,left, align = center] {(no exploration \\ allowed)};

\draw[-, thick](spottingoutput)--(postspottingoutput);
\draw[->, thick](postspottingoutput)--(FDPfinaloutputleft) node[midway,left, align = center] {exploration \\ along path};

\draw[-, blue, thick](GSoutput)--(postGSoutput);
\draw[->, blue, text = black, thick](postGSoutput)--(FDPfinaloutputright) node[midway,right, align = center] {arbitrary \\ exploration};

\end{tikzpicture}
}
\caption{\small Schematic for proposed FDP bounds (shaded gray nodes), in the context of the usual FDR control framework (nodes with red borders) and the GS closed testing framework for simultaneous FDP control (nodes with blue borders). The proposed bounds borrow the path construction from FDR procedures to leverage side information, while obtaining simultaneous guarantees like the GS approach to permit exploration.}
\label{fig:schematic}
\end{figure}

We prove our bounds by developing a simple yet versatile proof technique---based on a martingale argument rather different from those commonly used in the FDR literature---to obtain tight non-asymptotic bounds for the probability that the stochastic process $|\cR_k \cap \mathcal H_0|$ of false discoveries hits certain boundaries. This technique is inspired by the proof of FDR control for the multilayer knockoff filter \citep{katsevich17mkf}.

Several simultaneous bounds $\fdpbar$ for the sets $\cR_t \equiv \{j: p_j \leq t\}$ are already available \citep{genovese2006exceedance, meinshausen2006false, meinshausen2006estimating, BetR17, Hemerik2019}, in addition to bounds valid for all subsets \citep{van2004multiple, GS11mt}. However, all of these bounds treat p-values exchangeably, whereas the link we establish between FDR and simultaneous FDP control allows us to leverage the rich recent literature (e.g. \cite{LB17, lei2016adapt}) on incorporating side information into FDR procedures  to obtain powerful simultaneous FDP bounds. Importantly, our bounds apply also to the knockoffs procedure \citep{BC15, CetL16} for high-dimensional variable selection, which produces an ordered set of independent ``one-bit p-values" to which we may apply one of our bounds. They also apply to the online setting, where p-values come in a stream and decisions to accept or reject must be made before seeing future data. Our results can also be used as diagnostic tools for FDR procedures: one can run an FDR procedure at a certain level and then obtain a valid $1-\alpha$ confidence bound on the FDP of the resulting rejection set. Finally, all of our bounds \eqref{fdpbar} have an appealingly simple closed form. 

Next, we preview our simultaneous FDP bound for the knockoffs procedure and demonstrate its utility on a large genome-wide association study data set (Section~\ref{sec:real_data}). Then, we state our main results and provide a high level proof sketch in Section~\ref{sec:theoretical}. In Section~\ref{sec:connections_selective}, we compare and contrast our theoretical results with those in the FDR literature. We then compare the performance of our simultaneous FDP bounds with existing alternatives via numerical simulations (Section~\ref{sec:simulations}) and then conclude the paper in Section~\ref{sec:conclusion}. The code to reproduce our numerical simulations and data analysis is available online at \url{https://github.com/ekatsevi/simultaneous-fdp}. 

\section{An illustration with real data} \label{sec:real_data}

Before formally stating and proving our bounds, we first illustrate their utility in the context of an application to genome-wide association studies (GWAS). The goal of GWAS is to identify the genetic factors behind various human traits. For this purpose, genotype and trait data are collected from large cohorts of individuals and then scanned for association. The recently compiled UK Biobank resource \citep{BetJ18} has data on half a million individuals. 

GWAS represents a vast variable selection problem, with genotypes viewed as covariates and the trait as the outcome. Since nearby genotypes are strongly correlated with each other, the units of inference usually are spatially localized genomic regions instead of individual genetic variants (i.e. variables are grouped before testing). The \textit{knockoffs} framework \citep{BC15} for variable selection with FDR control has been proposed to analyze GWAS data \citep{SetC17} and has recently been applied to several phenotypes in the UK Biobank \citep{SetS19}. 

The knockoffs procedure falls into the class of FDR procedures introduced in the previous section. A set of \textit{knockoff statistics} $W_1, \dots, W_p$ are constructed for each group of genetic variants, with the property that the distribution of $W_k$ is symmetric about the origin for null groups $k$. On the other hand, knockoff statistics for non-null groups should be large and positive. Therefore, an ordering $\pi(1), \dots \pi(p)$ of the groups is constructed by sorting the knockoff statistics by decreasing magnitude. The $k$th rejection set $\cR_k$ along the path is then defined as the set of groups among the first $k$ in the ordering whose knockoff statistics have positive signs:
\begin{equation*}
\cR_k \equiv \{\pi(j) \leq k: \textnormal{sign}(W_{\pi(j)}) > 0\}.
\end{equation*}
FDR control is proved for regularization $a_0 = 1$ and the estimate
\begin{equation}
\vh(\cR_k) \equiv |\{\pi(j) \leq k:  \textnormal{sign}(W_{\pi(j)}) < 0\}|,
\label{vh_knockoffs}
\end{equation}
which leverages the sign-symmetry of $W_k$ for null $k$.

Our theoretical result (Corollary~\ref{cor:knockoffs} in Section~\ref{sec:theoretical}) shows that the bound 
\begin{equation*}
\overline V_{\text{knockoff}}(\cR_k) \equiv \left\lfloor\frac{\log(\alpha^{-1})}{\log(2-\alpha)} \cdot \left(1 + |\{\pi(j) \leq k: \textnormal{sign}(W_{\pi(j)}) < 0\}|\right) \right\rfloor,
\end{equation*}
i.e. expression~ \eqref{fdpbar} with $a = a_0 = 1$, $c(\alpha) = \frac{\log(\alpha^{-1})}{\log(2-\alpha)}$, and $\vh$ as in definition~\eqref{vh_knockoffs}, satisfies the uniform coverage statement \eqref{simultaneous_selective}. For $\alpha = 0.05$, we have $c(\alpha) \approx 4.5$. In other words, \textit{inflating the knockoffs FDP estimate by 4.5 allows us to upgrade from bounding the FDP of one set on average to confidently bounding the true FDP across the entire path}. 

To illustrate the utility of this result, we apply it to the analysis of the platelet count trait in the UK Biobank data, borrowing the knockoff statistics that were made publicly available by \cite{SetS19} at \url{https://msesia.github.io/knockoffzoom/ukbiobank.html}. While several correlation cutoffs were used to create groups in \cite{SetS19}, here we consider the lowest resolution groups (of average width 0.226 megabases), whose size corresponds roughly to that yielded by current GWAS methodologies. In Figure \ref{fig:GWAS}, we plot $\fdpbar$ and $\fdphat$ as a function of the rejection set size. The dashed line shows the FDR target level $q = 0.1$ used by \cite{SetS19}. The $\fdphat$ curve crosses this threshold at $k^*$ with $|\cR_{k^*}|= 1460$. By comparison, the $\fdpbar$ curve is (necessarily) more conservative, but clearly yields informative FDP bounds for many rejection sets. It crosses the line $q = 0.1$ at $|\cR_{k^*}| = 813$, meaning that we are 95\% confident that at least 90\% of the top 813 genomic loci are associated with platelet count.

\begin{figure}[h!]
	\centering
	\includegraphics[width = \textwidth]{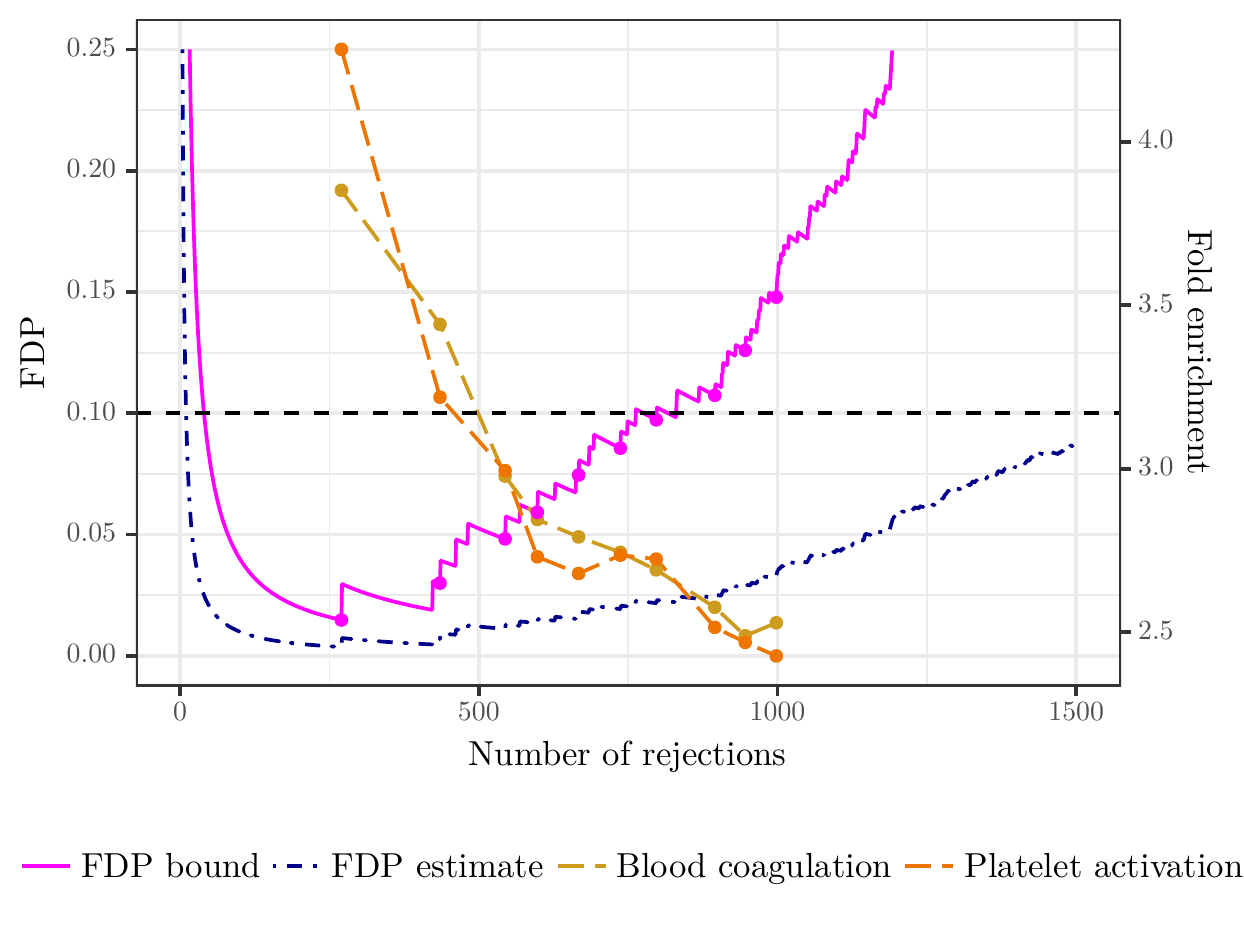}
	\caption{\small Knockoffs FDP estimate (dark blue) and proposed FDP upper bound (magenta) for GWAS analysis of platelet count, with highlighted points indicating interesting rejection sets. The degree of over-representation of two relevant Gene Ontology terms (orange and goldenrod) among genes in the neighborhood of genomic regions defined by each interesting rejection set.}
	\label{fig:GWAS} 
\end{figure}

Importantly, though, we can do much more than this. Instead of committing to $q = 0.1$ before seeing the data, we can explore several rejection sets along the knockoffs path, examining their content and FDP bound. One strategy we might take is to choose a set of points along the knockoffs path that represent different compromises between FDP bound and rejection set size; these points are highlighted along the $\fdpbar$ curve in Figure \ref{fig:GWAS}. For example, the leftmost highlighted point represents a rejection set of size 270 with an FDP bound of 0.015 (with 95\% confidence).


A domain expert might inspect each of these points and choose one that makes the most sense. In genetics, a common first step to evaluate a set of discoveries is to see whether they fit with known associations. The Gene Ontology, or GO \citep{AetE00} is a collection of biological processes, each annotated with a set of genes known to be involved in that process. Given a set of genomic regions, the GREAT \citep{MetB10} tool computes the ``enrichment" (i.e.  overrepresentation) of genes annotated to any given GO term falling in those genomic regions. For the platelet count trait, we would expect associated regions to be overrepresented for genes annotated to processes like ``blood coagulation" or ``platelet activation." We computed the fold enrichment (degree of overrepresentation) for these two terms, shown in Figure \ref{fig:GWAS} as dashed goldenrod and orange lines. The fold enrichment generally decreases as we increase the size of the rejection set, corresponding to our intuition that the strongest signals are generally in regions previously known to be associated with platelet count. By juxtaposing domain-specific annotations with simultaneous FDP bounds, a plot like Figure \ref{fig:GWAS} would already go a long way towards helping a domain expert decide on a biologically meaningful rejection set with a statistically sound Type-I error guarantee.


Finally, to quantify the price we pay for this extra flexibility, we consider several traits analyzed by \cite{SetS19} and compare the numbers of rejections we get for the original analysis ($\text{FDR} \leq 0.1$) with the numbers we get for controlling FDP at various levels based on $\fdpbar_{\text{knockoff}}$. The results are shown in Table \ref{table:real_data}. As we can see, there is certainly a trade-off between analytical flexibility and statistical power. However, at least in this dataset, we can still make substantial numbers of discoveries while enjoying the benefit of improved flexibility.

\begin{table}
	\begin{center}
	\includegraphics[width = \textwidth]{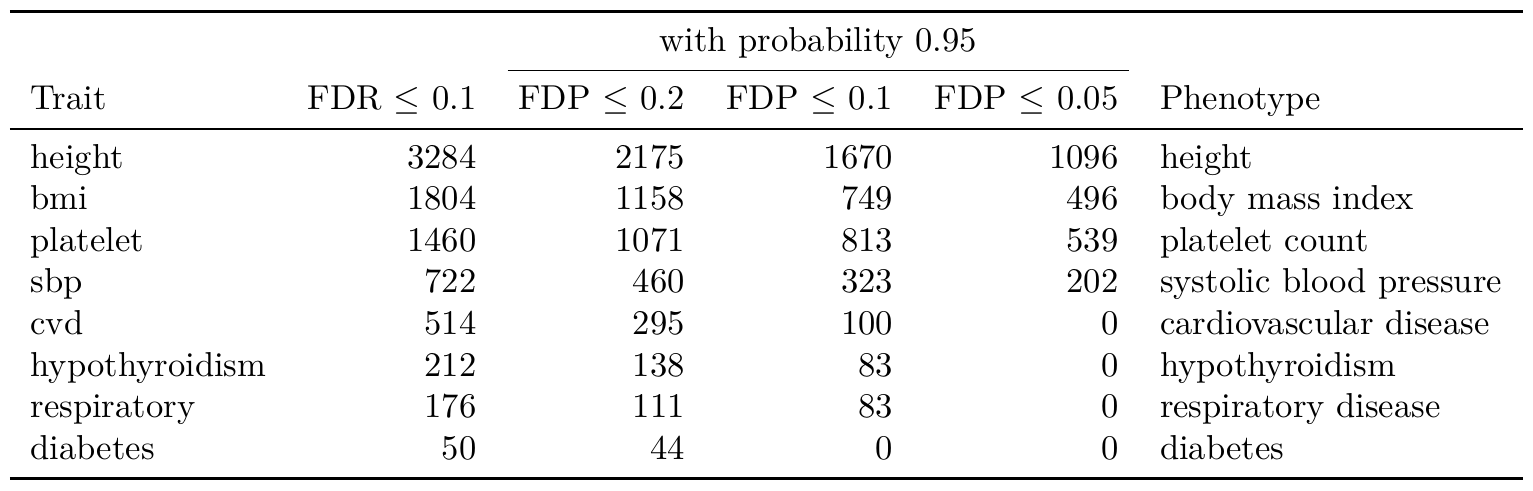}
	\end{center}
	\caption{Number of knockoffs discoveries for different traits in the UK Biobank, based on original analysis targeting $\fdr \leq 0.1$ and proposed bounds with three FDP thresholds.}
	\label{table:real_data} 
\end{table}

Having previewed the utility of our theoretical results on a real dataset, in the next section we formally state these results.

\section{Main results} \label{sec:theoretical}

In this section, we present a set of paths $\Pi$ along with corresponding bounds $\fdpbar$ of the form \eqref{fdpbar} and state conditions under which the guarantee \eqref{simultaneous_selective} holds. In fact, for brevity we will specify only the numerator $\overline V$ of $\fdpbar$. As discussed in the introduction, both $\Pi$ and $\vh$ will be borrowed directly from existing FDR procedures. We provide bounds for both the \textit{batch} and \textit{online} settings. In the batch setting, there is a finite number of hypotheses $H_1, \dots, H_n$ for which the p-values are available all at once; in the online setting, where there is an infinite stream of hypotheses, which arrive one at a time and a decision must be made about each hypothesis as soon as its p-value arrives. The proofs for all our results are provided in the supplement, but a sketch of the main idea is given in Section~\ref{sec:proof_sketch}.



\subsection{FDP bounds in the batch setting}

Here, we have a fixed, finite set of hypotheses $H_1, \dots, H_n$ and a set of p-values $p_1, \dots, p_n$. To construct a path, consider first ordering the hypotheses in some way $\pi(1), \pi(2), \dots, \pi(n)$, constructing $\pi$ to encourage non-nulls to appear near the beginning of the order. Then, define a p-value cutoff $p_* \in (0, 1]$. We form a path $\Pi$ by traversing the ordering and choosing hypotheses whose p-values passed the cutoff:
\begin{equation}
\Pi \equiv (\cR_0, \cR_1, \dots, \cR_n) \text{ such that } \cR_k \equiv \{\pi(j): j \leq k, p_{\pi(j)} \leq p_*\}.
\label{general_path_old}
\end{equation}
There are three ways of defining the path $\Pi$:
\begin{enumerate}
	\item $\textit{sort}$: $\pi$ is formed by sorting $p$-values; in this case usually $p_* \equiv 1$.
	\item $\textit{preorder}$: $\pi$ is fixed ahead of time using prior knowledge.
	\item $\textit{interact}$: $\pi$ is built on the fly using prior knowledge and $p$-values.
\end{enumerate}
Next, we elaborate on these path constructions in the batch setting and present FDP bounds for each of them.

\subsubsection{Sorted path} \label{sec:sorted}

Ordering hypotheses by p-value, $p_{\pi(k)} = p_{(k)}$, and setting $p_* = 1$ leads to
\begin{equation}
\cR_k = \{j: p_j \leq p_{(k)}\}.
\label{BH_rejections}
\end{equation}
This is the most common path construction among multiple testing procedures, serving as the basis for the Benjamini-Hochberg (BH) algorithm and many other step up/down algorithms (e.g., \cite{benjamini1999step, Gavrilov2009}). It is the obvious choice when no side information is available. \cite{Storey04} formulated the BH algorithm in terms of the FDR control paradigm described in Section~\ref{sec:new_class}, with $\vh_{\text{sort}}(\cR_k) \equiv n \cdot p_{(k)}$ and $a_0 = 0$:
\begin{equation}
\fdphat_{\text{sort}}(\cR_k) \equiv \frac{n \cdot p_{(k)}}{|\cR_k|}.
\label{FDP_hat_BH}
\end{equation}
The following theorem presents our FDP bounds~\eqref{fdpbar} for this path, based on the estimate $\vh_{\text{sort}}$.

\begin{theorem}
\label{BH_THEOREM} 
Let $\cR_k$ be defined via \eqref{BH_rejections}, and let
\begin{equation*}
\overline V_{\sort}(\cR_k) \equiv \left \lfloor\frac{\log (\frac 1 \alpha)}{\log\left(1 + \log(\frac 1 \alpha)\right)}\cdot(1 + n \cdot p_{(k)})\right \rfloor.
\end{equation*}
If the null p-values are independent and stochastically larger than uniform, i.e. $\PP{p_j \leq s}\leq s$ for all $j \in \H_0$ and $s \in [0,1]$, then the uniform bound \eqref{simultaneous_selective} holds for all $\alpha \in (0,0.31]$, i.e.
\begin{equation*}
\PP{\fdp(\cR_k) \leq \fdpbar_{\sort}(\cR_k) \textnormal{ for all } k \in [n]} \geq 1-\alpha.
\end{equation*}
\end{theorem}

\begin{remark} \label{BH_remark}
	In Theorem \ref{BH_THEOREM}, we require that $\alpha \leq 0.31$. However, strong numerical evidence shows that the bound is valid for all $\alpha$. The restriction on $\alpha$ is an artifact of our proof and does not represent an intrinsic breaking point of the bound. Despite this limitation in our proof, the range $\alpha \leq 0.31$ includes most confidence levels that would be used in practice (although the case $\alpha = 0.5$ might be of interest to bound the median of the FDP distribution and $\alpha = 1$ of interest to bound the null proportion).
\end{remark}

\subsubsection{Pre-ordered path}
The \textit{pre-ordered setting} applies when prior information (e.g. data from a similar experiment) sheds light on which hypotheses are more likely to be non-null, so a good ordering $\pi$ is known in advance. Several FDR methodologies taking advantage of pre-specified orderings have been developed; \cite{g2013false} and \cite{LB17} build paths using $p_* = 1$ while \cite{BC15} and \cite{lei2016power} use $p_* \in (0,1)$.

For the case $p_* = 1$, we use a construction from the accumulation test of \cite{LB17}: an \textit{accumulation function} $h$ is a function $h:[0,1] \rightarrow \mathbb R_+$ that is non-decreasing and integrates to 1. Then, we define 
\begin{equation}
\vh_{\textnormal{preorder-acc}}(\cR_k) \equiv \sum_{j = 1}^k h(p_{\pi(j)}).
\label{vh_preorder_acc}
\end{equation}
Alternatively, for $p^*\in(0,1)$, we can follow Selective SeqStep \citep{BC15} and Adaptive SeqStep \citep{lei2016power} to define 
\begin{equation}
\vh_{\preordersel}(\cR_k) \equiv \sum_{j = 1}^k \frac{p_*}{1-\lambda}I(p_{\pi(j)} > \lambda),
\label{vh_preorder_sel}
\end{equation}
where $\lambda \geq p_*$.
The following theorem presents our FDP bounds~\eqref{fdpbar} for the pre-ordered setting for the cases $p_* = 1$ and $p_* \in (0,1)$, which rely on estimates~\eqref{vh_preorder_acc} and \eqref{vh_preorder_sel}, respectively.

\begin{theorem}
	\label{ORDERED_THEOREM}
	Fix $a > 0$ and assume the null p-values are independent and stochastically larger than uniform. Given a prior ordering $\pi$, let 
	\begin{equation*}
	\cR_k \equiv \{\pi(j): j \leq k, p_{\pi(j)} \leq p_*\}.
	\end{equation*}
	
	\begin{enumerate}
		\item Set $p_* = 1$, choose a (possibly unbounded) accumulation function $h$, and define
		\begin{equation*}
		\overline V^h_{\textnormal{preorder-acc}}(\cR_k) \equiv \left\lfloor\frac{\log(\frac 1 \alpha)}{a\log\left(\left(\int_0^1 \alpha^{h(u)/a}du\right)^{-1}\right)} \cdot \left(a + \sum_{j = 1}^k h(p_{\pi(j)})\right)\right\rfloor.
		\label{acc_general_bound}
		\end{equation*}
		Then, the uniform bound \eqref{simultaneous_selective} holds for all $\alpha \in (0,1)$:
		\begin{equation*}
		\PP{\fdp(\cR_k) \leq \fdpbar^h_{\textnormal{preorder-acc}}(\cR_k)  \textnormal{ for all } k \in [n]} \geq 1-\alpha.
		\end{equation*}
		Moreover, if $\sup_{u \in [0,1]}h(u) \equiv B < \infty$, then we may instead use
		\begin{equation*}
		\overline V^B_{\textnormal{preorder-acc}}(\cR_k) \equiv \left \lfloor\frac{\log (\frac 1 \alpha)}{a\log\left(\left(1 - \frac{1-\alpha^{B/a}}{B}\right)^{-1}\right)}\cdot \left(a + \sum_{j = 1}^k h(p_{\pi(j)})\right)\right \rfloor.
		\label{acc_seqstep}
		\end{equation*}
		\item Set $p_* \in (0,1)$ and fix $\lambda \geq p_*$. Define 
		\begin{equation*}
		\overline V^B_{\textnormal{preorder-sel}} \equiv \left\lfloor\frac{\log(\frac 1 \alpha)}{a\log\left(1 + \frac{1-\alpha^{B/a}}{B}\right)}\cdot\left(a + \sum_{j = 1}^k \frac{p_*}{1-\lambda}I(p_{\pi(j)} > \lambda)\right)\right \rfloor,
		\label{c_AS}
		\end{equation*}
		where $B \equiv \frac{p_*}{1-\lambda}$. Then, uniform bound \eqref{simultaneous_selective} holds for all $\alpha \in (0,1)$:
		\begin{equation*}
		\PP{\fdp(\cR_k) \leq \fdpbar^B_{\textnormal{preorder-sel}}(\cR_k)  \textnormal{ for all } k \in [n]} \geq 1-\alpha.
		\end{equation*}
	\end{enumerate}
\end{theorem}

As we previewed in Section \ref{sec:real_data}, we can apply this bound to the knockoff filter \citep{BC15}, a variable selection methodology based on the idea of creating a \textit{knockoff variable} for each original variable, and then using these knockoffs as controls for the originals. Instead of p-values, the knockoff filter produces \textit{knockoff statistics} $W_j$ for each variable $j$. These are constructed so that
\begin{equation}
\{\textnormal{sign}(W_j)\}_{j \in \mathcal H_0} \independent \{|W_j|\}_{j \in [p]}, \{\textnormal{sign}(W_j)\}_{j \not \in \mathcal H_0}; \ \  \{\textnormal{sign}(W_j)\}_{j \in \mathcal H_0} \overset{\text{i.i.d.}}\sim \text{Ber}(1/2).
\label{knockoffs_property}
\end{equation}
The signs of the knockoff statistics are therefore a set of independent ``one-bit p-values'', to which the above theorem applies. 
\begin{corollary} \label{cor:knockoffs}
Let $W_1, \dots, W_p$ be a set of knockoff statistics satisfying property~\eqref{knockoffs_property}. Let $\pi$ be the ordering corresponding to sorting $W_j$ by decreasing magnitude, and define $\cR_k = \{\pi(j) \leq k: \textnormal{sign}(W_{\pi(j)}) > 0\}$. Then, bound~\eqref{simultaneous_selective} holds for 
\begin{equation*}
\overline V_{\textnormal{knockoff}}(\cR_k) \equiv \left\lfloor \frac{\log(\alpha^{-1})}{a\log(2-\alpha^{1/a})} \cdot \left(a + |\{\pi(j) \leq k: \textnormal{sign}(W_{\pi(j)}) < 0\}|\right)\right\rfloor.
\end{equation*}
\end{corollary}
\begin{proof}
Define $p_j = 1/2$ for $W_j < 0$ and $p_j = 1$ for $W_j > 0$. By property~\eqref{knockoffs_property}, it is easy to see that these p-values are independent of the ordering $\pi$ (so the ordering can be treated as fixed) and satisfy the assumptions of Theorem~\ref{ORDERED_THEOREM}. The rejection sets $\cR_k$ are defined via \eqref{general_path_old} with $p_* = 0.5$ and $\fdpbar$ is defined via \eqref{vh_preorder_sel} with $p_* = \lambda = 0.5$. Therefore, we may apply part 2 of Theorem~\ref{ORDERED_THEOREM}, plugging in $B = \frac{p_*}{1-\lambda} = 1$.
\end{proof}

\subsubsection{Interactive path}

In the \textit{interactive setting}, p-values are split into ``orthogonal" parts, with one part being used---together with side information---to determine a hypothesis ordering $\pi$ and the other part being used for FDR control. AdaPT \citep{lei2016adapt} uses the masked p-values $g(p_j) = \min(p_j, 1-p_j)$ and side information $x_j$ to build up the ordering, defining a path based on \eqref{general_path_old} with $p_* = 0.5$. It then uses $\vh_{\preordersel}$ with $p_* = \lambda = 0.5$ to construct an FDP estimate based on which the algorithm chooses a rejection set. This procedure is like Selective SeqStep, but with the ordering constructed interactively. STAR \citep{lei17star}, on the other hand, is the interactive analog of the accumulation test, using $p_* = 1$ and $\vh_{\text{preorder-acc}}$. It is shown that any bounded accumulation function $h$ has a corresponding orthogonal masking function $g$, based on which the ordering can be constructed. 

For our simultaneous FDP bounds, we use a slightly different path definition than AdaPT and STAR: we build up the path $\pi$ from beginning to end, while these two methods proceed in the opposite direction. However, we do not expect this change to impact the quality of the constructed path. The path construction we consider is as follows. $\pi(1)$ is chosen based on the information $\sigma(\{x_j, g(p_j)\}_{j \in [n]})$. Once $\pi(1)$ is chosen, the corresponding p-value $p_{\pi(1)}$ is unmasked, so the information $\sigma(\{x_j, g(p_j)\}_{j \in [n]}, p_{\pi(1)})$ can be used to choose $\pi(2)$. In general, we can choose $\pi(k+1)$ in any way based on the information
\begin{equation}
\mathcal G_k \equiv \sigma(\{x_j, g(p_j)\}_{j \in [n]}, \{p_{\pi(j)}\}_{j \leq k}).
\label{interactive_filtration}
\end{equation}
Therefore, as in AdaPT and STAR, the ordering $\pi$ may be built up interactively, with a human in the loop deciding the order based on $\mathcal G_k$. The following theorem provides FDP bounds for interactively constructed paths. 

\begin{theorem} \label{INTERACTIVE_THEOREM} 	
	Let $\pi$ be any ordering predictable with respect to the filtration~\eqref{interactive_filtration}, where $g$ is a masking function as defined below, and let
	\begin{equation*}
	\cR_k \equiv \{\pi(j): j \leq k, p_{\pi(j)} \leq p_*\}.
	\end{equation*}
	\begin{enumerate}
		\item Let $h$ be an accumulation function bounded by $B$ and let $g$ is its corresponding masking function (see \cite{lei17star}). Set $p_* = 1$, and define $\fdpbar^B_{\textnormal{interact-acc}} \equiv \fdpbar^B_{\textnormal{preorder-acc}}$. If the null p-values are independent of each other and of the non-null p-values, and the null p-values have non-decreasing densities, then uniform bound~\eqref{simultaneous_selective} holds for all $a > 0$ and all $\alpha \in (0,1)$:
		\begin{equation*}
		\PP{\fdp(\cR_k) \leq \fdpbar^B_{\textnormal{interact-acc}} (\cR_k)  \textnormal{ for all } k \in [n]} \geq 1-\alpha.
		\end{equation*}
		\item Fix $p_* \in (0,1)$ and $\lambda \geq p_*$. Define $g(p) = \min(p, \frac{p_*}{1-p_*}p)$ and $\fdpbar^B_{\textnormal{interact-sel}} \equiv \fdpbar^B_{\textnormal{preorder-sel}}$. If the null p-values are independent of each other and of the non-nulls, and the null p-values are mirror-conservative (see \cite{lei2016adapt}), then uniform bound~\eqref{simultaneous_selective} holds for all $\alpha \in (0,1)$:
		\begin{equation*}
		\PP{\fdp(\cR_k) \leq \fdpbar^B_{\textnormal{interact-sel}} (\cR_k)  \textnormal{ for all } k \in [n]} \geq 1-\alpha.
		\end{equation*}
	\end{enumerate}
\end{theorem}

These results are similar to the previous section's bounds, but are more subtle due to the data-dependent ordering $\pi$.

\subsection{FDP bounds for any online algorithm}

Now, we turn to FDP bounds for the online setting. In this setting, decisions about hypotheses must be made as they arrive one at a time in a stream. Moreover, the order in which hypotheses arrive might or might not be the in the experimenter's control. Therefore, non-nulls might not necessarily occur early, and further the rejection decision for the $H_k$ must be made without knowing the outcomes of future experiments. Hence, in general, online multiple testing procedures must proceed differently from  batch ones:
online procedures adaptively produce a sequence of levels $\alpha_j$ at which to test hypotheses. Assuming for simplicity that $\pi(j)=j$, these levels define the online path:
\begin{equation}
\label{online_path}
\Pi_{\online} \equiv (\cR_1, \cR_2, \dots \cR_n, \dots) \text{ where } \cR_k \equiv \{ j \leq k: p_j \leq \alpha_j \}.
\end{equation}
The levels $\alpha_j$ are chosen based on the outcomes of past experiments, i.e. 
\begin{equation}
\alpha_{k+1} \in \mathcal G_k \supseteq \sigma(I(p_j \leq \alpha_j); j \leq k).
\label{online_predictable}
\end{equation}

The alpha-investing procedure of \cite{FS08} and follow-up works \citep{aharoni2014generalized, JM16, RYWJ17} are built on the analogy of testing a hypothesis at level $\alpha_j$ as spending wealth. One pays a price to test each hypothesis, and is rewarded for each rejected hypothesis. For each of these methods, the levels $\alpha_j$ are adaptively constructed to ensure that the wealth always remains non-negative. In this paper, we consider paths of the form (\ref{online_path}) corresponding to arbitrary sequences $\{\alpha_j\}$ satisfying requirement (\ref{online_predictable}), including those constructed by existing algorithms but any others as well.

Until recently, online FDR methods were formulated without reference to any $\fdphat$. However, \cite{RYWJ17}  noted that LORD \citep{JM16} implicitly bounds $\fdphat(\cR_k)$ for $a_0 = 0$ and 
\begin{equation*}
\vh_{\simpleol}(\cR_k) \equiv \sum_{j = 1}^k \alpha_j.
\label{vh_simpleol}
\end{equation*}
They also used this fact to design a strictly more powerful algorithm called LORD++.
Moving beyond  LORD++, \cite{SAFFRON} proposed an adaptive algorithm called SAFFRON, which uses $a_0=0$ and
\begin{equation*}
\vh_{\adaptiveol}(\cR_k) \equiv \sum_{j = 1}^k \frac{\alpha_j}{1-\lambda_j} I(p_{j}>\lambda_j).
\label{vh_adaptiveol}
\end{equation*}

SAFFRON improves upon the LORD estimate by correcting for the proportion of nulls, making it the online analog of the Storey-BH procedure \citep{Storey04}. Like the levels $\alpha_j$, the constants $\lambda_j$ may also be chosen based on the outcomes of prior experiments. 

The following theorem provides FDP bounds~\eqref{fdpbar} corresponding to the above two choices for $\vh_\online$.

\begin{theorem}\label{ONLINE_THEOREM}
	Fix $a > 0$ and let $\alpha_1, \alpha_2, \dots$ be any sequence of thresholds predictable with respect to filtration $\mathcal G_k$, as in \eqref{online_predictable}. Suppose the null p-values are stochastically larger than uniform conditional on the past:
	\begin{equation}
	\PP{p_k \leq s|\mathcal G_{k-1}} \leq s \quad \text{for each } k \in \mathcal H_0 \text{ and each } s \in [0,1].
	\label{online_assumption}
	\end{equation}
	\begin{enumerate}
		\item Define
		\begin{equation*}
		\overline V_{\simpleol}(\cR_k)\equiv \left \lfloor\frac{\log(\frac 1 \alpha)}{a\log\left(1+\frac{\log(\frac 1 \alpha)}{a}\right)}\cdot\left(a + \sum_{j = 1}^k \alpha_j\right)\right\rfloor.
		\label{c_LORD}
		\end{equation*} 
		Then uniform bound \eqref{simultaneous_selective} holds for all $\alpha \in (0,1)$:
		\begin{equation*}
		\PP{\fdp(\cR_k) \leq \fdpbar_{\simpleol}(\cR_k) \textnormal{ for all } k \geq 0} \geq 1-\alpha.
		\end{equation*}
		\item  Let $\lambda_j \geq \alpha_j$ for all $j$, $\{\lambda_j\}$ be predictable with respect to $\mathcal G_k$, and $\sup_{j}\frac{\alpha_j}{1-\lambda_j} \equiv B < \infty$. Define 
		\[	
		\overline V^B_{\adaptiveol}(\cR_k) \equiv  \left\lfloor\frac{\log(\frac 1 \alpha)}{a\log\left(1 + \frac{1-\alpha^{B/a}}{B}\right)}\cdot\left(a + \sum_{j = 1}^k \frac{\alpha_j}{1-\lambda_j} I(p_{j}>\lambda_j)\right)\right\rfloor.
		\]
		Then, uniform bound (\ref{simultaneous_selective}) holds for all $\alpha \in (0,1)$:
		\begin{equation*}
		\PP{\fdp(\cR_k) \leq \fdpbar_{\adaptiveol}(\cR_k) \textnormal{ for all } k \geq 0} \geq 1-\alpha.
		\end{equation*}
	\end{enumerate} 
\end{theorem}


The closest existing result to Theorem \ref{ONLINE_THEOREM} is that of \cite{JM16} (JM). JM consider a truncated version of generalized alpha-investing rules that satisfy a uniform FDX bound like $\PP{\sup_{k} \text{FDP}_{k} \geq \gamma} \leq \alpha.$ Their result is similar in spirit to part 1 of Theorem \ref{ONLINE_THEOREM}, but there are subtle differences. Their results, like most other FDX bounds, are pre hoc, meaning that given a $\gamma, \alpha \in (0,1)$, their procedure produces a sequence of rejections satisfying the FDX guarantee. Our guarantees are post hoc, meaning that they would apply to any sequence of rejections produced by any online algorithm, that may or may not have been designed for FDR or FDP control.

\subsection{A glimpse of the proof} \label{sec:proof_sketch}

In this section, we present a key exponential tail inequality lemma (Lemma \ref{main_lemma}) that underlies the proofs of Theorems \ref{ORDERED_THEOREM}, \ref{INTERACTIVE_THEOREM}, and \ref{ONLINE_THEOREM}. The proof of Theorem \ref{BH_THEOREM} requires a more involved proof technique, which we defer to the supplementary materials (Section~\ref{app:BH}), where we also show how Theorems \ref{ORDERED_THEOREM}, \ref{INTERACTIVE_THEOREM}, and \ref{ONLINE_THEOREM} follow from Lemma \ref{main_lemma} below (Section~\ref{app:234}). We use a martingale-based proof technique that is distinct from the technique used to prove FDR control; see Section \ref{sec:comparing_proof_techniques} for a comparison.


\begin{lemma} \label{main_lemma}
	Consider a (potentially infinite) set of hypotheses $H_1, H_2, \dots$, an ordering $\pi(1), \pi(2), \dots$, and a set of cutoffs $\alpha_1, \alpha_2, \dots$. Let
	\begin{equation*}
	\cR_k \equiv \{j \leq k: p_{\pi(j)} \leq \alpha_j\} \quad \text{and} \quad \fdphat_a(\cR_k) \equiv \frac{a + \sum_{j \leq k} h_j(p_{\pi(j)})}{|\cR_k|},
	\end{equation*}
	where $\{h_j\}_{j \geq 1}$ are functions on $[0,1]$ and the subscript $a$ on $\fdphat_a$ makes the dependence on the regularization $a > 0$ explicit. Suppose there exists a filtration 
	\begin{equation}
	\F_k \supseteq \sigma(\H_0, \{\pi(j)\}_{j \leq k}, \{h_j(p_{\pi(j)}), I(p_j \leq \alpha_j)\}_{j\leq k, \pi(j) \in \H_0})
	\label{proof_filtration}
	\end{equation}
	such that $\text{for all } \pi(k) \in \H_0$, we have
	\begin{equation}
	\PPst{p_{\pi(k)} \leq \alpha_k}{\F_{k-1}} \leq \alpha_k ~\text{ and }~ \EEst{h_k(p_{\pi(k)})}{\F_{k-1}} \geq \alpha_k,
	\label{p_h_requirement}
	\end{equation}
	almost surely. Then, for each $x > 1$ and $a > 0$, 
	\begin{equation}
	\PP{\sup_{k \geq 0}\frac{\fdp(\cR_k)}{\fdphat_a(\cR_k)} \geq x} \leq \exp(-a\theta_x x),
	\label{general_FDP_bound}
	\end{equation}
	where $\theta_x$ is defined in the following four cases:
	\begin{enumerate}
		\item If $h_k = h$ for some accumulation function $h$, $\alpha_k = 1$, $\pi(k)$ is pre-specified (i.e. nonrandom), and $p_{\pi(k)} \independent \F_{k-1}$ for all $\pi(k) \in \H_0$, then 
		$\theta_x$ is the unique positive root of the equation
		\begin{equation}
		\int_0^1 \exp(-\theta x h(u))du = \exp(-\theta).
		\label{theta_eq_acc}
		\end{equation}
		\item If $h_k = h$ for some accumulation function $h$ bounded by $B$ and $\alpha_k = 1$, then 
		$\theta_x$ is the unique positive root of the equation
		\begin{equation}
		\exp(-\theta) + \frac{1-\exp(-\theta x B)}{B} = 1.
		\label{theta_eq_B}
		\end{equation}
		\item If $h_k(p) = 0$ for all $p \leq \alpha_k$, and $h_k(p) \leq B$ for all $k,p$, then 
		$\theta_x$ is the unique positive root of the equation
		\begin{equation}
		\exp(\theta) - \frac{1-\exp(-\theta x B)}{B} = 1.
		\label{theta_eq_AS}
		\end{equation}
		\item If $h_k(p_k) = \alpha_k$, then 
		$\theta_x$ is the unique positive root of the equation 
		\begin{equation}
		e^\theta = 1 + \theta x.
		\label{theta_eq}
		\end{equation}
	\end{enumerate}
\end{lemma}

Let us outline the proof of the lemma. Fix any arbitrary $x > 1$ and $\theta > 0$. We first restrict our attention to only the nulls as follows:
\begin{equation*}
\begin{split}
&\PP{\sup_{k \geq 0}\frac{\fdp(\cR_k)}{\fdphat_a(\cR_k)} \geq x} \quad=\quad \PP{\sup_{k \geq 0}\frac{V(\cR_k)}{a + \vh(\cR_k)} \geq x} \\
&\quad= \PP{\sum_{j = 1}^k I(p_{\pi(j)} \leq \alpha_j)I(\pi(j) \in \H_0) \geq ax + x\sum_{j = 1}^k h_j(p_{\pi(j)}), \text{ for some } k \geq 0} \\
&\quad\leq \PP{\sum_{j = 1}^k I(p_{\pi(j)} \leq \alpha_j)I(\pi(j) \in \H_0) \geq ax + x\sum_{j = 1}^k h_j(p_{\pi(j)})I(\pi(j) \in \H_0), \text{ for some } k \geq 0}.
\end{split}
\end{equation*}
Now, we may rearrange terms and employ the Chernoff exponentiation trick, to conclude that:
\begin{equation*}
\begin{split}
&\PP{\sup_{k \geq 0}\frac{\fdp(\cR_k)}{\fdphat_a(\cR_k)} \geq x} \\
&\quad= \PP{\sup_{k \geq 0}\ \exp\left(\theta\left(\sum_{j = 1}^k [I(p_{\pi(j)}\leq \alpha_j)-xh_j(p_{\pi(j)})]I(\pi(j) \in \H_0) \right)\right)\geq \exp(a\theta x)} \\
&\quad\equiv \PP{\sup_{k \geq 0}\ Z_k \geq \exp(a\theta x)}.
\end{split}
\end{equation*}
We claim that if $\theta = \theta_x$, then $Z_k$ is a supermartingale with respect to $\F_k$. If this is the case, then the conclusion of the lemma would follow from the \cite{ville_etude_1939} maximal inequality for positive supermartingales:
\begin{equation*}
\PP{\sup_{k \geq 0}\ Z_k \geq \exp(a\theta x)} \leq \exp(-a\theta x)\EE{Z_0} = \exp(-a\theta x).
\label{maximal_inequality}
\end{equation*}
as desired. Hence, what remains is to show that in each of the four cases, the choices of $\theta_x$ make $Z_k$ a supermartingale. To derive the FDP bounds in Theorems~\ref{ORDERED_THEOREM}, \ref{INTERACTIVE_THEOREM}, and \ref{ONLINE_THEOREM}. we set
\begin{equation}
\fdpbar(\cR_k) \equiv x \cdot \fdphat_a(\cR_k),
\label{fdpbar_def}
\end{equation}
where $x$ is chosen such that $\exp(-a\theta_x x) = \alpha$. We defer these derivations to Section \ref{app:234} in the supplementary materials. In fact, our definition \eqref{fdpbar} has an added floor function in the numerator, which we may add for free because the true number of false discoveries $V(\cR_k)$ is always an integer. 

\section{Comparisons to work on FDR control}
\label{sec:connections_selective}

The paths and FDP bounds we construct are closely tied to existing FDR control algorithms. Table~\ref{table:ordered_testing} shows each of our bounds as well as the FDR methods they are related to.  In this section, we explore the relationships between our results and those already existing in the FDR literature.

\begin{table}[h!]
	\begin{center}
		\begin{tabular}{l|l|l|l}
			ordering & p-val cutoffs & $\vh(\cR_k)$ & FDR method\\
			\hline 
			sort& $p_* = 1$ & $n \cdot p_{(k)}$ & BH \\
			preorder   & $p_* = 1$ &  $\sum_{j \leq k}h(p_{j})$ & Accumulation test\\
			preorder & $ p_* \in (0,1)$ & $\sum_{j \leq k} \frac{p_*}{1-\lambda}I(p_j>\lambda)$ & Selective and Adaptive SeqStep \\ 		
			interact &$ p_* = 1$ & $\sum_{j \leq k}h(p_{\pi(j)})$ & STAR\\
			interact  &$p_*\in (0,1)$ & $\sum_{j \leq k}\frac{p_*}{1-\lambda}I(p_{\pi(j)}>\lambda)$ & AdaPT\\
			(online) & $\alpha_j \in \mathcal G_{j-1}$ & $\sum_{j \leq k} \alpha_j$ & LORD, LORD++\\
			(online) & $\alpha_j \in \mathcal G_{j-1}$ & $\sum_{j \leq k} \frac{\alpha_j}{1-\lambda_j}I(p_j>\lambda_j)$ & SAFFRON, alpha-investing \\
		\end{tabular}
	\end{center}
	\caption{Overview of proposed FDP bounds and FDR procedures inspiring them \newline ($h$ denotes an accumulation function).}
	\label{table:ordered_testing}
\end{table}

\subsection{Comparing the roles of $\fdphat$}

We start by recalling the definition \eqref{fdphat} of $\fdphat$. Batch FDR algorithms use this estimate of FDP to automatically choose the rejection set $\cR^* \in \Pi$, which is done via (\ref{k_star}). On the other hand, we use a regularized $\fdphat_a$ as a building block for our confidence envelopes $\fdpbar$ (recall definitions \eqref{fdpbar} and \eqref{fdpbar_def}), which the user may then inspect to choose $\cR^*$. It is important to remark here that while our bounds are \emph{inspired} by existing FDR algorithms, they are not intrinsically tied to the use of those procedures in any way. Indeed, we often employ the path $\Pi$ of known FDR procedures, but not their stopping criterion or choice of final rejected set.

Each FDR algorithm comes with a ``built-in" choice of regularization $a_0$. For example, the BH algorithm uses no regularization (i.e. $a_0 = 0$), while accumulation tests \citep{LB17} use $a_0 = \sup_{u \in [0,1]}h(u)$. The built-in regularizations are chosen to ensure FDR control (see below). On the other hand, we consider arbitrary regularizations $a > 0$, with different regularizations leading to different constants $c(\alpha)$ and therefore different confidence envelopes. Different regularization parameters lead to envelopes that are tighter in different places; we have found $a = 1$ to be a good choice.

\subsection{Comparing proof techniques} \label{sec:comparing_proof_techniques}
For each existing batch FDR algorithm, FDR control is established using the following martingale argument. First, the ratio $\frac{\fdp(\cR_k)}{\fdphat(\cR_k)} = \frac{V(\cR_k)}{a_0 + \vh(\cR_k)}$ is upper-bounded by a stochastic process $L_k$, such that $L_k$ is a supermartingale with respect to a \textit{backwards} filtration $\{\Omega_k\}_{k = n, \dots, 1}$. Furthermore, it is shown that $\EE{L_n} \leq 1$. The choice of regularization $a_0$ is usually made to ensure the existence of such an $L_k$. Using the fact that $k^*$ picked using rule \eqref{k_star} is a stopping time with respect to $\Omega_k$, we obtain $\fdr = \EE{\fdp(\cR_{k^*})} \leq q$ in a single line using the optional stopping theorem:
\begin{equation*}
\label{one-liner}
\frac{\EE{\fdp(\cR_{k^*})}}{q} \leq \EE{\frac{\fdp(\cR_{k^*})}{\fdphat(\cR_{k^*})}} \leq \EE{L_{k^*}} \leq 1.
\end{equation*}
This technique was first used by \cite{Storey04} for the BH procedure, but remarkably, the other batch procedures mentioned in this paper like knockoffs, AdaPT, STAR, ordered tests and others, all implicitly use the same argument (though it was not expressed as succinctly as above), each with different $L_k,\Omega_k$. While we also rely on a martingale argument to prove our FDP bounds (recall Section \ref{sec:proof_sketch}), the martingales we construct are fundamentally different: they are exponential and employ forward filtrations instead of backwards ones.

Note that the original supermartingales $(L_k, \Omega_k)$ used to prove FDR control for batch procedures can also be used to obtain tail bounds like \eqref{general_FDP_bound}, for original regularization $a = a_0$. Indeed, using Ville's maximal inequality again, we find
\begin{equation*}
\PP{\sup_{0 \leq k \leq n} \frac{\fdp(\cR_{k})}{\fdphat(\cR_{k})} \geq x} \leq \PP{\sup_{0 \leq k \leq n} L_k \geq x} \leq \frac{1}{x}\EE{L_n} \leq \frac{1}{x}.
\label{unregularized_inequalities}
\end{equation*}
Therefore, for each batch procedure we consider $\fdpbar(\cR_k) = \frac{1}{\alpha}\fdphat(\cR_k)$ is also a valid upper confidence band for FDP. Versions of this bound have been considered before in the case of BH, e.g. by \cite{robbins1954one} and \cite{goeman2016simultaneous}. This implies that for all considered batch procedures, we have
\begin{equation*}
\text{Median}\left[\sup_{\cR \in \Pi} \frac{\fdp(\cR)}{\fdphat(\cR)}  \right] \leq 2.
\label{median_bound}
\end{equation*}
However, note that the constants $c(\alpha) = \alpha^{-1}$ grow quickly as $\alpha$ decays. On the other hand, the constants we provide scale logarithmically, rather than linearly, in $\alpha^{-1}$. 

\subsection{Comparing assumptions} \label{sec:comparing_assumptions}
This martingale argument for FDR control and the argument we employ here both require some form of independence among the p-values. Furthermore, our assumptions for each of these theorems are identical to or weaker than the ones needed to prove FDR control. For Theorem \ref{BH_THEOREM}, we only need to make assumptions on the distribution $(p_j)_{j \in \H_0}$, so unlike existing proofs of FDR control for BH, we do not make any assumptions on the dependence of the nulls on the non-nulls (see \cite{Dwork2018} for another example of such a result). In Theorem \ref{ORDERED_THEOREM}, we assume that the nulls are independent and stochastically larger than uniform, whereas for the original FDR control results \citep{BC15, LB17} it was also required that nulls be independent of non-nulls. Furthermore, part 1 of Theorem \ref{ORDERED_THEOREM} provides an FDP bound for possibly unbounded accumulation functions, whereas the original work proposing accumulation tests \citep{LB17} requires accumulation functions to be bounded. In Theorems \ref{INTERACTIVE_THEOREM} and \ref{ONLINE_THEOREM}, our assumptions are identical to those in the original works. Finally, we remark that the only FDR procedure which has a guarantee under dependence is BH, for which a non-martingale proof was proposed by \citep{BY01}.

Next, we illustrate the performance of some of our bounds in simulations.

\section{Numerical simulations} \label{sec:simulations}

In this section, we compare the proposed FDP bounds to existing bounds, in the sorted and pre-ordered settings. We also examine the effect of correlation on the proposed bounds. In all cases, we take $n = 2500$ and $\alpha = 0.1$. For the proposed bounds, we take $a = 1$.

\subsection{Sorted setting} \label{sec:sorted_simulations}

As discussed in the introduction, the setting in which the most prior work has been done is when hypotheses are ordered based on p-value. In other words, we are concerned with bounds $\fdpbar$ for the sets $\cR_t \equiv \{j: p_j \leq t\}$ such that
\begin{equation}
\PP{\fdp(\cR_t) \leq \fdpbar(\cR_t) \ \text{ for all } t \in [0,1]} \geq 1-\alpha.
\label{BH_path}
\end{equation}

\subsubsection{Comparing to other explicit, finite-sample bounds}

The bounds most comparable to ours are explicit, finite-sample bounds. Two such bounds were proposed by \cite{meinshausen2006estimating}: $\fdpbar(\cR_t) \equiv \frac{\overline V(t)}{|\cR_t|}$ for
\begin{equation*}
\overline V_{\text{Robbins}}(t) \equiv \frac{1}{\alpha}nt; \quad \overline V_{\text{DKW}}(t) \equiv \sqrt{\frac{n}{2}\log \frac{1}{\alpha}} + nt.
\end{equation*}
These bounds derive from inequalities by \cite{robbins1954one} and \cite{dvoretzky1956asymptotic}, respectively. Note that inequality \eqref{BH_path} for $\overline V_{\text{DKW}}$ is based on the one-sided DKW inequality and is valid for $\alpha < 0.5$. Compare these to our bound, which is
\begin{equation*}
\overline V_\sort(t) \equiv \frac{\log (\frac 1 \alpha)}{\log\left(1 + \log(\frac 1 \alpha)\right)}\cdot (1 + nt).
\end{equation*}
Note that we may add the floor function to all three of these bounds, omitted here for simplicity. By inspecting these three bounds, we see that the DKW bound is the tightest when $t$ is large, the Robbins bound is the tightest when $t$ is small, and our bound is the tightest in an intermediate range. 

\begin{figure}[h!]
	\includegraphics[width = \textwidth]{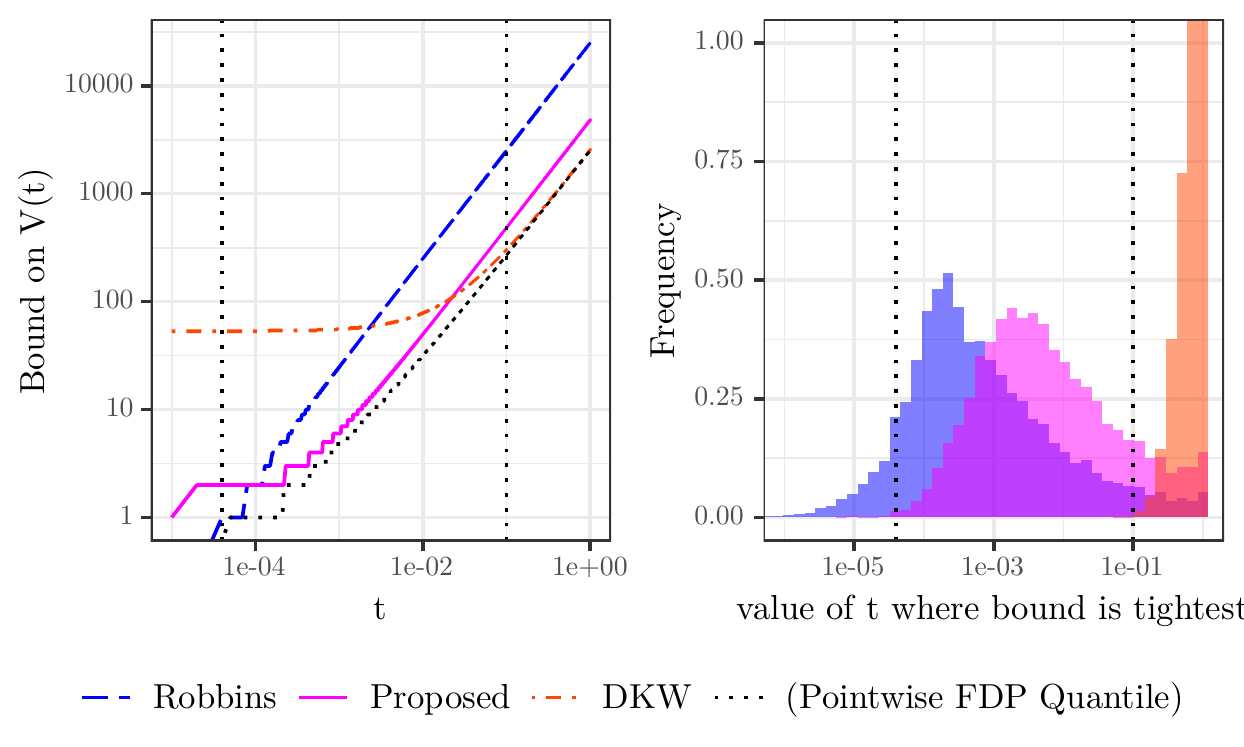}
	\caption{\small Comparing proposed bounds to two other explicit finite-sample bounds in the sorted setting. Vertical dotted lines indicate Bonferroni level $\alpha/n$ and nominal level $\alpha$. Left panel: The three FDP bounds; the proposed bound is tighter than the other two across most of the interesting range. Right panel: histograms of $t$ where the bound~\eqref{BH_path} is tightest.} 
	\label{fig:spotting_comparison}
\end{figure}

The left panel of Figure~\ref{fig:spotting_comparison} compares the three bounds (floor functions included), with dotted vertical lines indicating the Bonferroni level and the nominal level, respectively. The interval between these two levels is often the most interesting for multiple testing purposes, and the proposed bound is the tightest over most of this range (in particular, it is tighter than the Robbins bound as long as $\overline V_{\text{Robbins}}(t) \geq 2.4$). In fact, the proposed bound is not too far from the pointwise $1-\alpha$ quantile of $V(t)$, which is plotted for reference in black in the left panel. The right panel of Figure \ref{fig:spotting_comparison} shows a histogram of the value of $t$ at which the bound \eqref{BH_path} is tightest. We see that the majority of the time (about 87\%), our bound is tightest in the interesting range.

\subsubsection{Comparison to GS bound} \label{sec:simulations_BH_GS}

As discussed in the introduction, the GS bound is based on a suite of local tests $\{\phi_\cR\}_{\cR \in 2^{\H}}$. Therefore, different bounds can be obtained for different local tests. Here, we compare the proposed bound to the GS bound based on the Simes and Fisher local tests. We note that the Simes local test rejects if and only if the Robbins bound is nontrivial for any $t$. In fact, the GS-Simes bound is the closure of the Robbins bound and therefore dominates it \citep{GHS19}, so we remove the latter from consideration in this section.

\begin{figure}[h!]
	\centering
	\includegraphics[width = \textwidth]{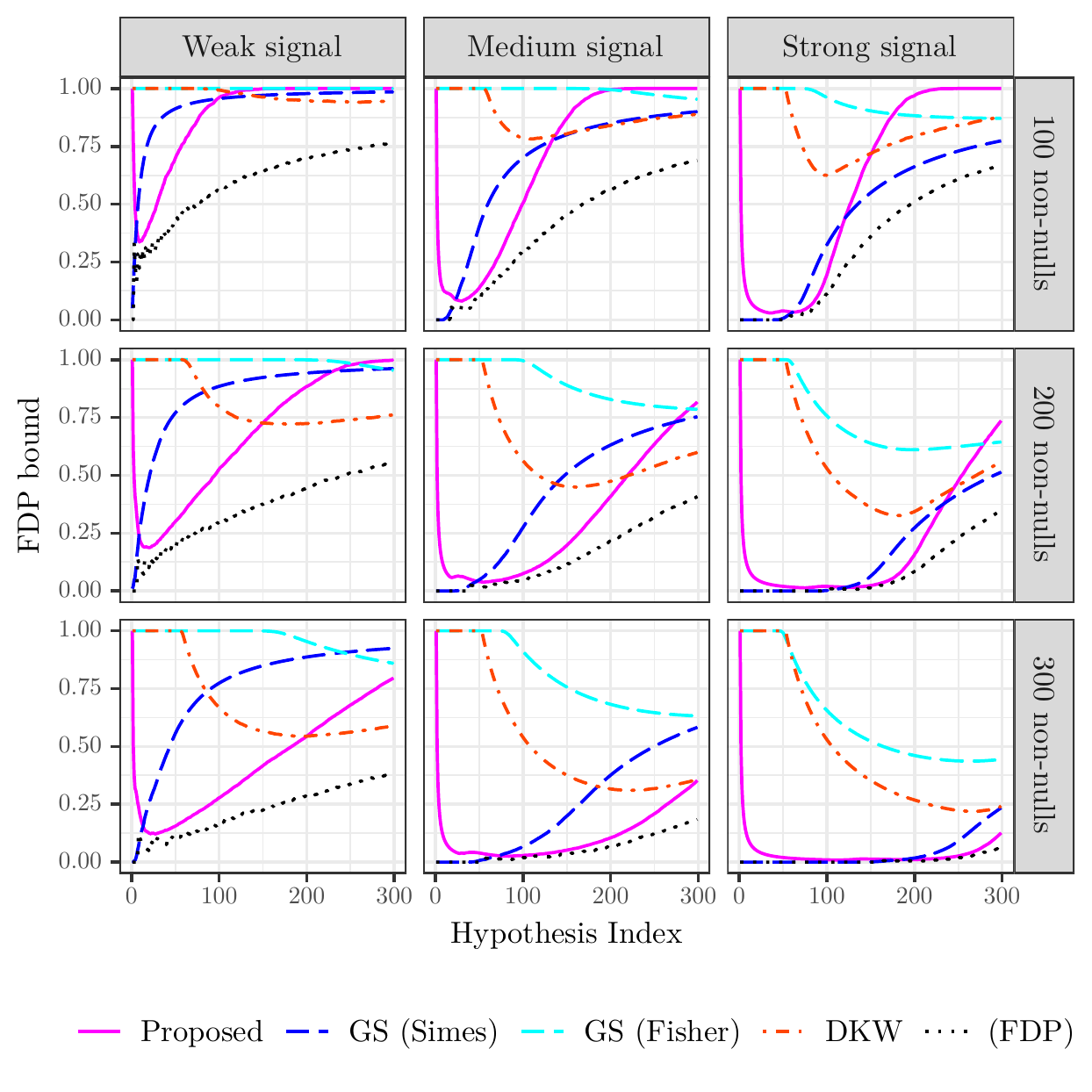}
	\caption{\small Comparing the proposed FDP bound with the GS bound (based on Simes or Fisher local tests) and the DKW bound in the sorted setting. The $1-\alpha$ quantile of the true FDP is also shown. The panels correspond to the three signal strengths and numbers of non-nulls. The proposed bounds are tightest in an intermediate range of rejection set sizes.}
	\label{fig:BH_simulation} 
\end{figure}

Since the GS bound is not explicit, we must make the comparison by inspecting the average shape of $\fdpbar$ on simulated data. We simulate independent test statistics $X_j \sim N(\mu_j, 1)$, where $\mu_j = \mu I(j \in \mathcal H_1)$ for some signal strength $\mu > 0$ and set of non-nulls $\mathcal H_1$. We then compute one-side p-values $p_j = 1 - \Phi(X_j)$. To cover a broad range of data-generating distributions, we consider the values $\mu = 2,3,4$ (weak, medium, and strong signal) and $|\mathcal H_1| = 100, 200, 300$.

Figure \ref{fig:BH_simulation} shows the average $\fdpbar$ curves (over 100 repetitions) in each of the nine simulation scenarios for the proposed and GS bounds, as well as the DKW bound introduced before. For reference, the $1-\alpha$ quantile of the true FDP is also shown. We see that the GS bounds inherit the properties of their underlying local tests. The GS-Simes bound behaves like the Robbins bound: it is tightest for small rejection set sizes, yielding highly nontrivial bounds near the beginning of the path for most simulation scenarios. The GS-Fisher bound behaves the opposite way: it is tightest for large rejection set sizes, even more so than the DKW bound. Neither the GS-Fisher bound nor the DKW bound yield very informative bounds in most of the simulation settings considered. Finally, the proposed bound is an intermediate between these two extremes, yielding the tightest estimates for intermediate rejection set sizes. For the simulation settings considered, the proposed bounds are tightest in interesting regions of the path: where many rejections are made but the FDP bound is still fairly low (e.g. below 0.2). 

\subsubsection{Coverage properties of FDP estimate} \label{sec:BH_simulation_overshoot}

The estimate $\fdphat_{\text{sort}}(t) = m \cdot t/|\cR_t|$ from equation~\eqref{FDP_hat_BH} and the related q-value, both proposed by \cite{Storey04}, have been shown to have asymptotic uniform coverage properties. In particular, their Theorem 6 states that for all $\delta > 0$,
\begin{equation}
\lim_{n \rightarrow \infty} \inf_{t \geq \delta}\left\{\fdphat(t) - \fdp(t)\right\} \geq 0 \quad \text{with probability 1}.
\label{Storey_bound}
\end{equation}
At first glance, this result might suggest that there is no reason to use conservative bounds for the FDP, if asymptotically, the much smaller point estimate bounds the FDP across the entire path. However, such a conclusion is misleading. Note that the infimum in the bound \eqref{Storey_bound} excludes $t \in [0, \delta)$, so for the bound to be interesting the value of $\delta$ must be small. Furthermore, the convergence becomes slower as $\delta \rightarrow 0$, so in finite samples the FDP estimate might undershoot the true FDP at some points along the path. As a counterpoint to \eqref{Storey_bound}, consider the following two results, holding as long as the null p-values are uniform and independent, and $\frac{|\mathcal H_0|}{n} \geq \epsilon > 0$:
\begin{equation*}
\EE{\sup_{t \in [0,1]}\frac{\fdp(t)}{\fdphat(t)}} = \infty; \quad \limsup_{n \rightarrow \infty}\ \sup_{t \in [\frac{c}{n},1]}\frac{\fdp(t)}{\fdphat(t)} = \infty \ \text{almost surely}.
\end{equation*}
The first result is due to \cite{robbins1954one} and holds for any finite $n$, and the second is due to \cite{Wellner1978} and holds for any fixed $c \geq 0$. Therefore, $\fdphat$ can underestimate $\fdp$ by large factors if we remove the restriction on $t$ or lower-bound it from below by an $O(1/n)$ term. 

We investigate via simulation whether applying BH at various target levels without correction for exploration has any consequences in finite samples. In the simulation setting from the previous section with $\mu = 4$ and $|\mathcal H_1| = 100$, consider applying BH using nominal levels $q \in \mathcal Q$ for some set $\mathcal Q \subseteq [0,1]$. If $\fdp_{\text{BH}}(q)$ denotes the realized FDP of running BH at level $q$, then the quantity
\begin{equation*}
\max_{q \in \mathcal Q} \frac{\fdp_{\text{BH}}(q)}{q}
\end{equation*}
measures the maximum extent to which the realized FDP exceeds the nominal level. In Figure~\ref{fig:BH_simulation_overshoot}, we show the mean and upper 90\% quantile of the above quantity for $\mathcal Q = [q_{\min}, 1]$, with $q_{\min}$ taking a range values. The practitioner may be willing to restrict her attention to a smaller set of nominal levels chosen in advance, e.g. $\mathcal Q_0 = \{0.01,0.025,0.05,0.075,0.1,0.125, 0.15,0.175,0.2\}$. The individual points in the figure correspond to this choice. 

\begin{figure}[h!]
	\centering
	\includegraphics[width = 0.9\textwidth]{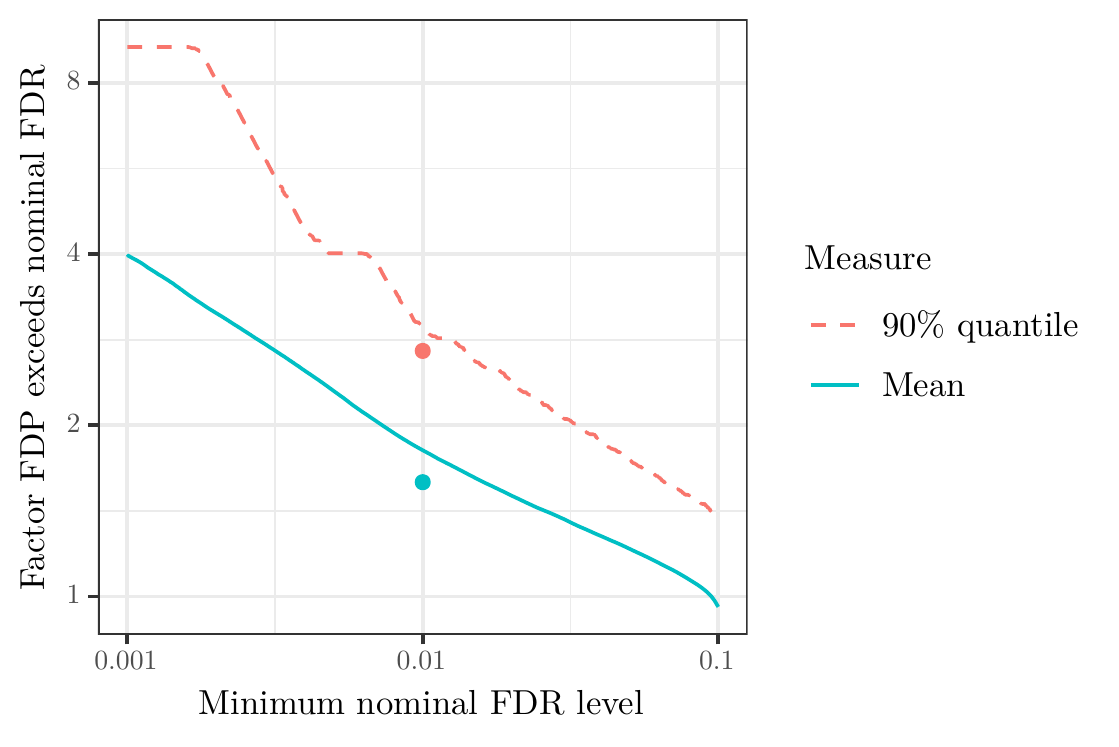}
	\caption{\small The extent to which the true FDP can exceed the nominal FDR level if BH is run for all nominal levels $q \in [q_{\min}, 1]$. Points correspond to running BH for the smaller prechosen set $\mathcal Q_0 = \{0.01,0.025,\dots, 0.2\}$ of FDR levels. We observe that some correction for exploration is necessary, whether the mean or the upper quantile of the FDP is of interest.}
	\label{fig:BH_simulation_overshoot}
\end{figure}	

Figure~\ref{fig:BH_simulation_overshoot} shows that the true FDP can significantly exceed the nominal level, a less comforting result than \eqref{Storey_bound}. Note that Theorem~\ref{BH_THEOREM} covers the case $q_{\min} = 0$ and gives a bound on the quantile of the FDP, so the corrections we introduce should be compared to the left-most point of the red dashed curve in Figure~\ref{fig:BH_simulation_overshoot}. On the other hand, this figure does show that the less we allow ourselves to explore, the smaller a price needs to be paid. We see this from the monotonically decreasing trend of the curves as a function of $q_{\min}$, and from the fact that the set $\mathcal Q_0 \subseteq [0.01,1]$ results in smaller factors than $[0.01,1]$. Deriving theoretical bounds for FDP under these kinds of limitations on exploration is an interesting direction for future work. Nevertheless, even when we restrict exploration in the above ways, we see it is dangerous to take the nominal FDR level at face value as suggested by statement \eqref{Storey_bound}. 

We note in passing that there has also been work on providing confidence envelopes such that the bound (\ref{BH_path}) holds asymptotically (for all $t$) as $n \rightarrow \infty$, e.g. by \cite{genovese2004stochastic} and \cite{meinshausen2006estimating}. However, we do not review these here for the sake of brevity.

\subsection{Pre-ordered setting} \label{sec:preordered_simulations}

Next, we consider the pre-ordered setting (fixing $\pi(j) = j$ without loss of generality). Here, we fix the number of non-nulls at $|\mathcal H_1| = 100$ and instead vary the degree to which the non-nulls tend to occur near the beginning of the ordering. We sample the non-nulls without replacement from $[n]$ according to a distribution with probability mass function proportional to the density of an exponential random variable with rate $\theta/n$. The greater $\theta$ is, the more informative the ordering is. We consider $\theta = 15, 35, 55$ (weak, medium, and strong ordering) and $\mu = 2,3,4$ (weak, medium, and strong signal, as before). Here, the DKW and Robbins bounds are not applicable, so we only compare to GS-Simes and GS-Fisher. We apply our bound based on $\vh_{\text{preorder-acc}}$, with accumulation function $h(p) = \frac{1}{1-\lambda}I(p > \lambda)$ with $\lambda$ = 0.1. We use the definition~\eqref{c_AS} of $\fdpbar_{\text{preorder-acc}}^B$.

Figure \ref{fig:ordered_simulation} shows the results. We see that the proposed bound effectively leverages the ordering information to obtain tighter FDP bounds than the GS-based methods. Predictably, the stronger the ordering information, the greater the advantage of our bound. Consistent with the previous simulation, GS-Simes outperforms GS-Fisher; the latter bound is nearly trivial for all simulation settings. Of course, an interesting direction of future work is to derive tighter GS-style bounds for settings with prior information.

\begin{figure}[h!]
	\centering
	\includegraphics[width = \textwidth]{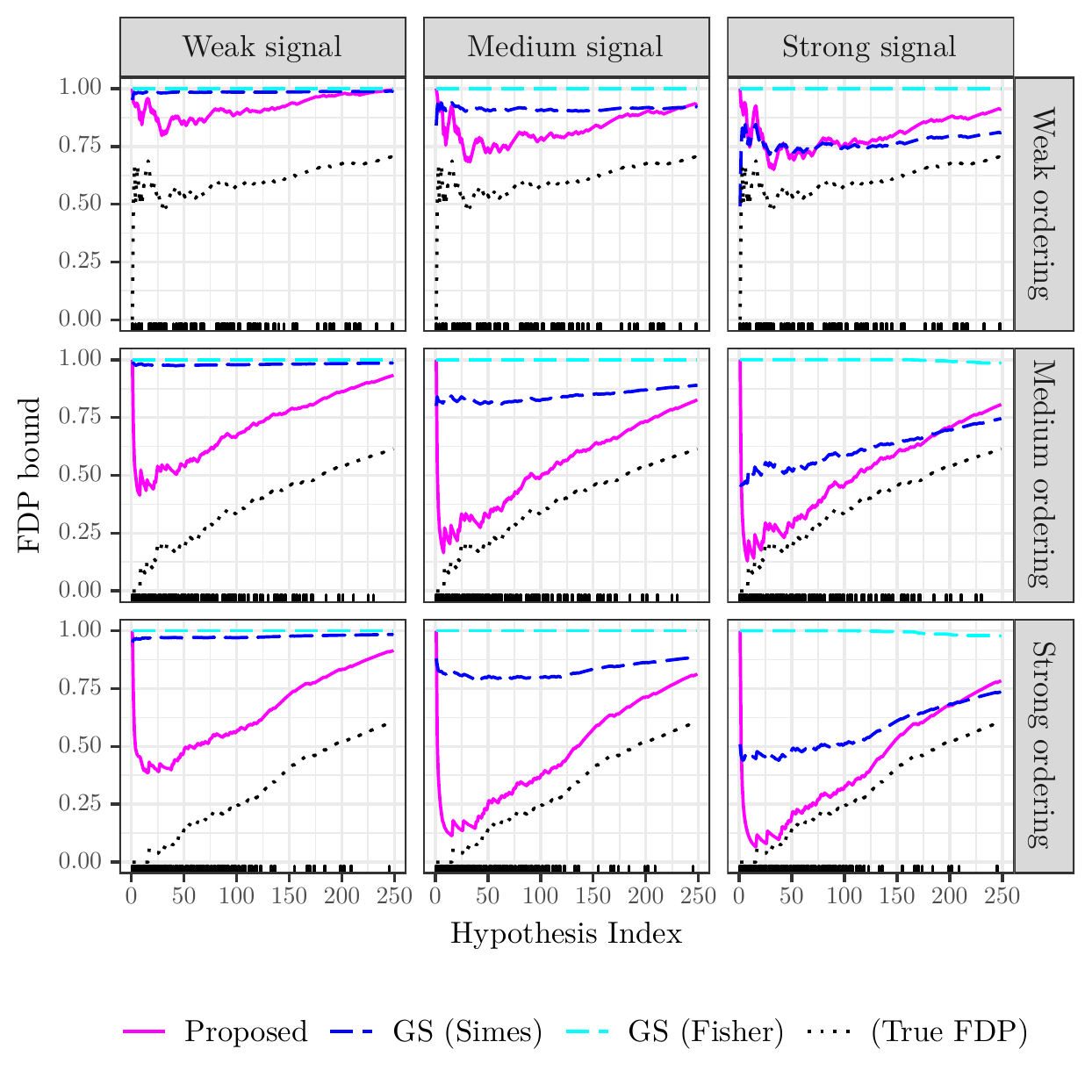}
	\caption{\small Comparing the proposed FDP bound with the GS bound based on Simes or Fisher local tests in the pre-ordered setting. The $1-\alpha$ quantile of the true FDP is also shown. The panels correspond to the three signal strengths and degrees to which non-nulls occur near the beginning of the ordering. Non-nulls are shown in the rug plots at the bottom of each panel. The proposed bounds leverage the ordering information to boost power.}
	\label{fig:ordered_simulation} 
\end{figure}

\subsection{The effect of correlation} \label{sec:simulations_correlation}

Finally, note that all our FDP bounds rely on some notion of independence among the p-values. Many of the FDR procedures considered here also only have guarantees under independence, though BH is a notable exception. Aside from online testing applications, independent p-values are hard to come by in practice, so more robust guarantees are necessary. BH is known to control FDR under the PRDS criterion \citep{BY01}, a form of positive dependence that contains no information about the \textit{strength} of the dependence. However, it is known that while the mean of FDP might not change much as dependence increases, the variance of the FDP will increase \citep{owen2005variance, EfronBook}. Hence, high-probability bounds on FDP under dependence are likely to use criteria other than PRDS to capture this dependence.

In this section, we use simulations to examine the extent to which our bounds continue to hold in the presence of p-value correlation. To model correlation, we draw the test statistics $X_j$ from an AR(1) process parameterized by correlation $\rho = -0.9, -0.8, \dots, 0.8, 0.9$. We consider four representative settings: the sorted setting from Theorem \ref{BH_THEOREM} and Section \ref{sec:sorted_simulations}, the pre-ordered setting with $p_* = 1$ from Theorem \ref{ORDERED_THEOREM} part 1 and Section \ref{sec:preordered_simulations}, the pre-ordered setting with $p_* = \lambda = 0.1$ from Theorem \ref{ORDERED_THEOREM} part 2, and the online setting with $\alpha_j = 0.05$ for all $j$ from Theorem \ref{ONLINE_THEOREM} part 1. For each setting and each value of $\rho$, we compute the $1-\alpha$ quantile of $\max_k \fdp(\cR_k)/\fdpbar(\cR_k)$, the maximum extent to which FDP can exceed our bound. We operate under the global null, since this is the worst case scenario.

\begin{figure}[h!]
	\centering
	\includegraphics[width = \textwidth]{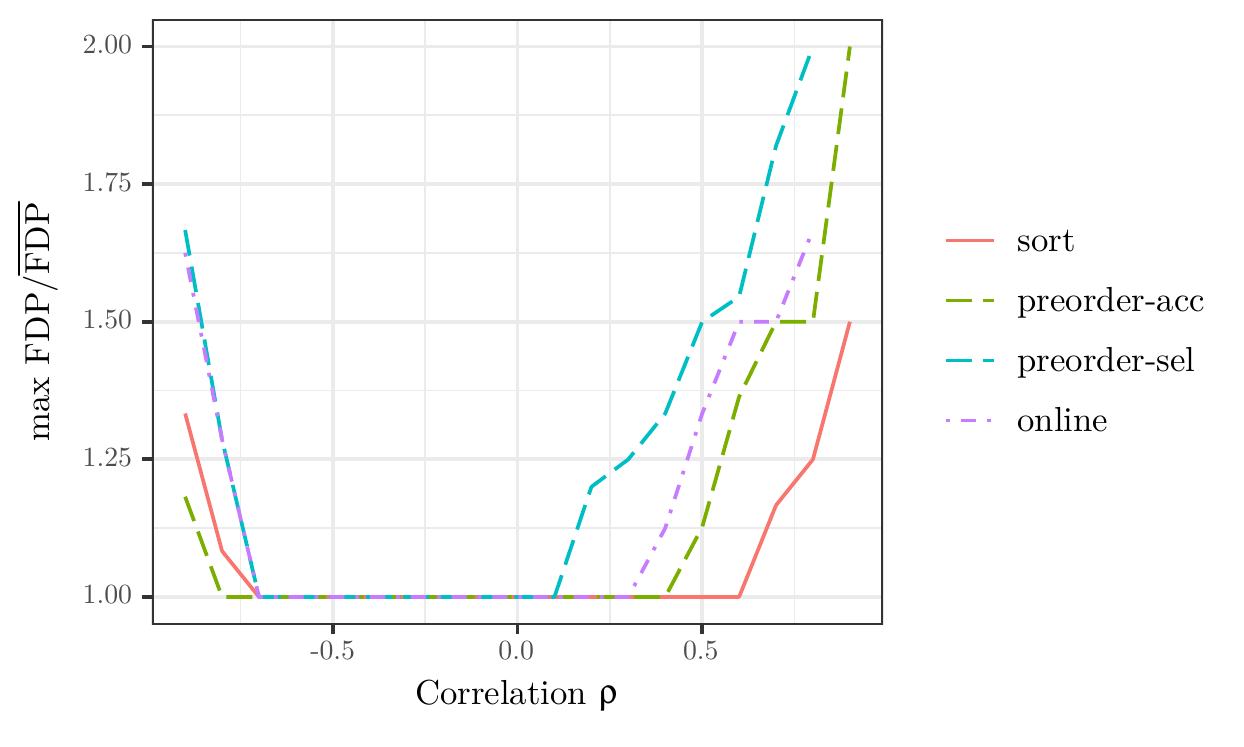}
	\caption{\small The extent to which FDP can exceed $\fdpbar$ for the proposed bounds under p-value correlation, generated from an AR(1) model parameterized by $\rho$. The bounds are more tolerant of negative than positive correlation.}
	\label{fig:correlation_simulation} 
\end{figure}

Figure \ref{fig:correlation_simulation} shows the simulation results. Reassuringly, all curves pass through 1 at $\rho = 0$, the independent case covered by our theorems. We see that different bounds have different tolerances for correlation, but negative correlation is tolerated better than positive correlation. All bounds continue to hold for $\rho \in [-0.7, 0.1]$. The bound in the sorted setting is particularly robust, continuing to be valid for $\rho \in [-0.7, 0.6]$. Nevertheless, it is not surprising that all the bounds are no longer valid once the correlation becomes strong enough. Indeed, under strong correlations the variability of the FDP necessitates more conservative bounds. We leave the extension of our results to the correlated setting for future work.

\section{Conclusion} \label{sec:conclusion}

In this paper, we establish a novel bridge between the realms of FDR control and simultaneous FDP control. While FDR procedures rely on estimates of the FDP to choose one rejection set from a path, we repurpose these estimates to obtain closed form simultaneous bounds on the FDP that are valid across the entire path with high probability. These novel bounds allow for the kind of simultaneous inference proposed by \cite{GS11mt}, where users can obtain FDP bounds on rejection sets they choose after exploring the data. They offer added versatility, applying in the structured, regression, and online settings; in Section~\ref{sec:simulations} we found that our bounds effectively leverage side information to boost power. 

Like any other simultaneous inference methodology, the bounds we provide must necessarily be conservative at certain points along the path. This reflects the fundamental trade-off between exploration and inference: allowing more flexibility to explore necessitates conservative corrections for inference to remain valid. In applications where analytical flexibility is important, however, this price may be worth paying. By augmenting the recent knockoffs analysis of the UK Biobank data set \citep{SetS19} with simultaneous FDP bounds, we saw in Section~\ref{sec:real_data} how much extra freedom we gained to find a biologically meaningful set of associations between genomic regions and human traits. While there might be a price in power as compared to the usual FDR analysis for other datasets, for this real-world dataset we still obtained meaningful FDP bounds for large rejection sets. Having said that, we do not necessarily advocate for employing simultaneous inference in all situations; indeed, FDR or FDX control at pre-specified levels may well be the right analysis choice in a variety of applications.

Figure~\ref{fig:BH_simulation_overshoot} illustrates the aforementioned trade-off between exploration and inference, and suggests that restricting the collection of rejection sets allowed to be explored can reduce the price paid for exploration. Studying this trade-off may lead to interesting future work. In addition to the fact that scientists may in practice explore fewer rejection sets than the guarantee covers, the rejection set they ultimately choose is likely not going to be the worst one in terms of FDP. Therefore, the worst-case bounds considered in this paper have this inherent degree of looseness. However, this looseness seems difficult or impossible to address theoretically. 

Recently, \cite{GHS19} explored the question of optimality among simultaneous inference procedures, proposing a natural admissibility criterion. In addition, these authors proved that only closed testing procedures, i.e. those of the kind proposed by \cite{GS11mt}, are admissible. Given any simultaneous inference procedure, like those proposed here, they showed how to improve the procedure by ``closing" it. From this perspective, our results can be viewed as building blocks from which to construct more sophisticated closed testing based procedures. It is not always the case that a closed testing procedure can be implemented in polynomial time, however, so it is still not clear which simultaneous bounds are dominated by other \textit{computationally efficient} bounds.

As pointed out to us by a referee, our bounds may also be used to construct new tests of the global null. Moreover, following \cite{meinshausen2006estimating}, our uniform bounds can also be used to estimate the null proportion among a set of hypotheses. Exploring these consequences of the proposed bounds is an interesting direction for future work. 

Finally, the proof technique we developed in this paper is versatile enough to cover a large portion of the currently available FDR procedures. Importantly, this includes the knockoffs procedure for variable selection in high dimensions. Like \cite{genovese2004stochastic}, we employ a stochastic process approach to analyze the FDP. However, while GW's bounds are asymptotic, we have used martingale arguments instead to obtain tight, non-asymptotic bounds. Perhaps these proof techniques may be extended further to apply to other multiple testing scenarios as well.

\section{Acknowledgements}

E.K. thanks Chiara Sabatti for her generous and valuable feedback on this work and on the manuscript itself, as well as David Siegmund, Emmanuel Candes, Anya Katsevich, and Michael Sklar for helpful discussions. A.R. acknowledges fruitful conversations with Ohad Feldheim, Jim Pitman and Jon Wellner about the BH empirical process. Both authors are also grateful for the insightful comments of the referees and AE. We acknowledge the UK Biobank Resource (application 27837) for the data used in Section~\ref{sec:real_data}, and Matteo Sesia for making public the summary statistics from the knockoffs analysis of this data \citep{SetS19}.

\begin{supplement}[id=suppA]
	\sname{Supplement}
	\stitle{}
	\slink[doi]{COMPLETED BY THE TYPESETTER}
	\sdatatype{.pdf}
	\sdescription{Proofs of all theorems.}
\end{supplement}

\bibliography{library}

\begin{thebibliography}{42}

\bibitem[\protect\citeauthoryear{Aharoni and
  Rosset}{2014}]{aharoni2014generalized}
\begin{barticle}[author]
\bauthor{\bsnm{Aharoni},~\bfnm{E}\binits{E.}} \AND
  \bauthor{\bsnm{Rosset},~\bfnm{S}\binits{S.}}
(\byear{2014}).
\btitle{{Generalized alpha-investing: definitions, optimality results and
  application to public databases}}.
\bjournal{Journal of the Royal Statistical Society: Series B (Statistical
  Methodology)}
\bvolume{76}
\bpages{771--794}.
\end{barticle}
\endbibitem

\bibitem[\protect\citeauthoryear{Anderson
  et~al.}{1995}]{anderson1995inequalities}
\begin{barticle}[author]
\bauthor{\bsnm{Anderson},~\bfnm{G~D}\binits{G.~D.}},
  \bauthor{\bsnm{Barnard},~\bfnm{R~W}\binits{R.~W.}},
  \bauthor{\bsnm{Richards},~\bfnm{K~C}\binits{K.~C.}},
  \bauthor{\bsnm{Vamanamurthy},~\bfnm{M~K}\binits{M.~K.}} \AND
  \bauthor{\bsnm{Vuorinen},~\bfnm{M}\binits{M.}}
(\byear{1995}).
\btitle{{Inequalities for zero-balanced hypergeometric functions}}.
\bjournal{Transactions of the American Mathematical Society}
\bpages{1713--1723}.
\end{barticle}
\endbibitem

\bibitem[\protect\citeauthoryear{Ashburner et~al.}{2000}]{AetE00}
\begin{barticle}[author]
\bauthor{\bsnm{Ashburner},~\bfnm{Michael}\binits{M.}},
  \bauthor{\bsnm{Ball},~\bfnm{Catherine~A}\binits{C.~A.}},
  \bauthor{\bsnm{Blake},~\bfnm{Judith~A}\binits{J.~A.}},
  \bauthor{\bsnm{Botstein},~\bfnm{David}\binits{D.}},
  \bauthor{\bsnm{Butler},~\bfnm{Heather}\binits{H.}},
  \bauthor{\bsnm{Cherry},~\bfnm{J~Michael}\binits{J.~M.}},
  \bauthor{\bsnm{Davis},~\bfnm{Allan~P}\binits{A.~P.}},
  \bauthor{\bsnm{Dolinski},~\bfnm{Kara}\binits{K.}},
  \bauthor{\bsnm{Dwight},~\bfnm{Selina~S}\binits{S.~S.}},
  \bauthor{\bsnm{Eppig},~\bfnm{Janan~T}\binits{J.~T.}} \AND
  \bauthor{\bsnm{Others}}
(\byear{2000}).
\btitle{{Gene Ontology: tool for the unification of biology}}.
\bjournal{Nature Genetics}
\bvolume{25}
\bpages{25}.
\end{barticle}
\endbibitem

\bibitem[\protect\citeauthoryear{Barber and Cand{\`{e}}s}{2015}]{BC15}
\begin{barticle}[author]
\bauthor{\bsnm{Barber},~\bfnm{Rina~Foygel}\binits{R.~F.}} \AND
  \bauthor{\bsnm{Cand{\`{e}}s},~\bfnm{Emmanuel~J}\binits{E.~J.}}
(\byear{2015}).
\btitle{{Controlling the false discovery rate via knockoffs}}.
\bjournal{The Annals of Statistics}
\bvolume{43}
\bpages{2055--2085}.
\end{barticle}
\endbibitem

\bibitem[\protect\citeauthoryear{Benjamini and Hochberg}{1995}]{BH95}
\begin{barticle}[author]
\bauthor{\bsnm{Benjamini},~\bfnm{Y}\binits{Y.}} \AND
  \bauthor{\bsnm{Hochberg},~\bfnm{Y}\binits{Y.}}
(\byear{1995}).
\btitle{{Controlling the false discovery rate: a practical and powerful
  approach to multiple testing}}.
\bjournal{Journal of the Royal Statistical Society: Series B (Statistical
  Methodology)}
\bvolume{57}
\bpages{289--300}.
\end{barticle}
\endbibitem

\bibitem[\protect\citeauthoryear{Benjamini and Liu}{1999}]{benjamini1999step}
\begin{barticle}[author]
\bauthor{\bsnm{Benjamini},~\bfnm{Yoav}\binits{Y.}} \AND
  \bauthor{\bsnm{Liu},~\bfnm{Wei}\binits{W.}}
(\byear{1999}).
\btitle{{A step-down multiple hypotheses testing procedure that controls the
  false discovery rate under independence}}.
\bjournal{Journal of Statistical Planning and Inference}
\bvolume{82}
\bpages{163--170}.
\end{barticle}
\endbibitem

\bibitem[\protect\citeauthoryear{Benjamini and Yekutieli}{2001}]{BY01}
\begin{barticle}[author]
\bauthor{\bsnm{Benjamini},~\bfnm{Y}\binits{Y.}} \AND
  \bauthor{\bsnm{Yekutieli},~\bfnm{D}\binits{D.}}
(\byear{2001}).
\btitle{{The control of the false discovery rate in multiple testing under
  dependency}}.
\bjournal{The Annals of Statistics}
\bvolume{29}
\bpages{1165--1188}.
\end{barticle}
\endbibitem

\bibitem[\protect\citeauthoryear{Blanchard, Neuvial and Roquain}{2017}]{BetR17}
\begin{barticle}[author]
\bauthor{\bsnm{Blanchard},~\bfnm{G}\binits{G.}},
  \bauthor{\bsnm{Neuvial},~\bfnm{P}\binits{P.}} \AND
  \bauthor{\bsnm{Roquain},~\bfnm{E}\binits{E.}}
(\byear{2017}).
\btitle{{Post hoc inference via joint family-wise error rate control}}.
\bjournal{arXiv}.
\end{barticle}
\endbibitem

\bibitem[\protect\citeauthoryear{Bycroft et~al.}{2018}]{BetJ18}
\begin{barticle}[author]
\bauthor{\bsnm{Bycroft},~\bfnm{Clare}\binits{C.}},
  \bauthor{\bsnm{Freeman},~\bfnm{Colin}\binits{C.}},
  \bauthor{\bsnm{Petkova},~\bfnm{Desislava}\binits{D.}},
  \bauthor{\bsnm{Band},~\bfnm{Gavin}\binits{G.}},
  \bauthor{\bsnm{Elliott},~\bfnm{Lloyd~T}\binits{L.~T.}},
  \bauthor{\bsnm{Sharp},~\bfnm{Kevin}\binits{K.}},
  \bauthor{\bsnm{Motyer},~\bfnm{Allan}\binits{A.}},
  \bauthor{\bsnm{Vukcevic},~\bfnm{Damjan}\binits{D.}},
  \bauthor{\bsnm{Delaneau},~\bfnm{Olivier}\binits{O.}},
  \bauthor{\bsnm{O'Connell},~\bfnm{Jared}\binits{J.}} \AND
  \bauthor{\bsnm{Others}}
(\byear{2018}).
\btitle{{The UK Biobank resource with deep phenotyping and genomic data}}.
\bjournal{Nature}
\bvolume{562}
\bpages{203}.
\end{barticle}
\endbibitem

\bibitem[\protect\citeauthoryear{Cand{\`{e}}s et~al.}{2018}]{CetL16}
\begin{barticle}[author]
\bauthor{\bsnm{Cand{\`{e}}s},~\bfnm{Emmanuel}\binits{E.}},
  \bauthor{\bsnm{Fan},~\bfnm{Yingying}\binits{Y.}},
  \bauthor{\bsnm{Janson},~\bfnm{Lucas}\binits{L.}} \AND
  \bauthor{\bsnm{Lv},~\bfnm{Jinchi}\binits{J.}}
(\byear{2018}).
\btitle{{Panning for gold: `model-X' knockoffs for high dimensional controlled
  variable selection}}.
\bjournal{Journal of the Royal Statistical Society: Series B (Statistical
  Methodology)}
\bvolume{80}
\bpages{551--577}.
\end{barticle}
\endbibitem

\bibitem[\protect\citeauthoryear{Dvoretzky, Kiefer and
  Wolfowitz}{1956}]{dvoretzky1956asymptotic}
\begin{barticle}[author]
\bauthor{\bsnm{Dvoretzky},~\bfnm{Aryeh}\binits{A.}},
  \bauthor{\bsnm{Kiefer},~\bfnm{Jack}\binits{J.}} \AND
  \bauthor{\bsnm{Wolfowitz},~\bfnm{Jacob}\binits{J.}}
(\byear{1956}).
\btitle{{Asymptotic minimax character of the sample distribution function and
  of the classical multinomial estimator}}.
\bjournal{The Annals of Mathematical Statistics}
\bpages{642--669}.
\end{barticle}
\endbibitem

\bibitem[\protect\citeauthoryear{Dwork, Su and Zhang}{2018}]{Dwork2018}
\begin{barticle}[author]
\bauthor{\bsnm{Dwork},~\bfnm{Cynthia}\binits{C.}},
  \bauthor{\bsnm{Su},~\bfnm{Weijie~J}\binits{W.~J.}} \AND
  \bauthor{\bsnm{Zhang},~\bfnm{Li}\binits{L.}}
(\byear{2018}).
\btitle{{Differentially Private False Discovery Rate Control}}.
\bjournal{arXiv}.
\end{barticle}
\endbibitem

\bibitem[\protect\citeauthoryear{Efron}{2010}]{EfronBook}
\begin{bbook}[author]
\bauthor{\bsnm{Efron},~\bfnm{B}\binits{B.}}
(\byear{2010}).
\btitle{{Large-scale inference}}.
\bseries{Institute of Mathematical Statistics (IMS) Monographs}
\bvolume{1}.
\bpublisher{Cambridge University Press, Cambridge}.
\end{bbook}
\endbibitem

\bibitem[\protect\citeauthoryear{Foster and Stine}{2008}]{FS08}
\begin{barticle}[author]
\bauthor{\bsnm{Foster},~\bfnm{D~P}\binits{D.~P.}} \AND
  \bauthor{\bsnm{Stine},~\bfnm{R~A}\binits{R.~A.}}
(\byear{2008}).
\btitle{alpha-investing: a procedure for sequential control of expected false
  discoveries}.
\bjournal{Journal of the Royal Statistical Society: Series B (Statistical
  Methodology)}
\bvolume{70}
\bpages{429--444}.
\end{barticle}
\endbibitem

\bibitem[\protect\citeauthoryear{Gavrilov, Benjamini and
  Sarkar}{2009}]{Gavrilov2009}
\begin{barticle}[author]
\bauthor{\bsnm{Gavrilov},~\bfnm{Yulia}\binits{Y.}},
  \bauthor{\bsnm{Benjamini},~\bfnm{Yoav}\binits{Y.}} \AND
  \bauthor{\bsnm{Sarkar},~\bfnm{Sanat~K.}\binits{S.~K.}}
(\byear{2009}).
\btitle{{An adaptive step-down procedure with proven FDR control under
  independence}}.
\bjournal{Annals of Statistics}
\bvolume{37}
\bpages{619--629}.
\bdoi{10.1214/07-AOS586}
\end{barticle}
\endbibitem

\bibitem[\protect\citeauthoryear{Genovese and
  Wasserman}{2004}]{genovese2004stochastic}
\begin{barticle}[author]
\bauthor{\bsnm{Genovese},~\bfnm{C}\binits{C.}} \AND
  \bauthor{\bsnm{Wasserman},~\bfnm{L}\binits{L.}}
(\byear{2004}).
\btitle{{A stochastic process approach to false discovery control}}.
\bjournal{The Annals of Statistics}
\bpages{1035--1061}.
\end{barticle}
\endbibitem

\bibitem[\protect\citeauthoryear{Genovese and
  Wasserman}{2006}]{genovese2006exceedance}
\begin{barticle}[author]
\bauthor{\bsnm{Genovese},~\bfnm{C~R}\binits{C.~R.}} \AND
  \bauthor{\bsnm{Wasserman},~\bfnm{L}\binits{L.}}
(\byear{2006}).
\btitle{{Exceedance control of the false discovery proportion}}.
\bjournal{Journal of the American Statistical Association}
\bvolume{101}
\bpages{1408--1417}.
\end{barticle}
\endbibitem

\bibitem[\protect\citeauthoryear{Goeman, Hemerik and Solari}{2019}]{GHS19}
\begin{barticle}[author]
\bauthor{\bsnm{Goeman},~\bfnm{Jelle}\binits{J.}},
  \bauthor{\bsnm{Hemerik},~\bfnm{Jesse}\binits{J.}} \AND
  \bauthor{\bsnm{Solari},~\bfnm{Aldo}\binits{A.}}
(\byear{2019}).
\btitle{{Only Closed Testing Procedures are Admissible for Controlling False
  Discovery Proportions}}.
\bjournal{arXiv}.
\end{barticle}
\endbibitem

\bibitem[\protect\citeauthoryear{Goeman and Solari}{2011}]{GS11mt}
\begin{barticle}[author]
\bauthor{\bsnm{Goeman},~\bfnm{J}\binits{J.}} \AND
  \bauthor{\bsnm{Solari},~\bfnm{A}\binits{A.}}
(\byear{2011}).
\btitle{{Multiple testing for exploratory research}}.
\bjournal{Statistical Science}
\bpages{584--597}.
\end{barticle}
\endbibitem

\bibitem[\protect\citeauthoryear{Goeman et~al.}{2016}]{goeman2016simultaneous}
\begin{barticle}[author]
\bauthor{\bsnm{Goeman},~\bfnm{J}\binits{J.}},
  \bauthor{\bsnm{Meijer},~\bfnm{R}\binits{R.}},
  \bauthor{\bsnm{Krebs},~\bfnm{T}\binits{T.}} \AND
  \bauthor{\bsnm{Solari},~\bfnm{A}\binits{A.}}
(\byear{2016}).
\btitle{{Simultaneous Control of All False Discovery Proportions in Large-Scale
  Multiple Hypothesis Testing}}.
\bjournal{arXiv}.
\end{barticle}
\endbibitem

\bibitem[\protect\citeauthoryear{G'Sell et~al.}{2016}]{g2013false}
\begin{barticle}[author]
\bauthor{\bsnm{G'Sell},~\bfnm{M~G}\binits{M.~G.}},
  \bauthor{\bsnm{Wager},~\bfnm{S}\binits{S.}},
  \bauthor{\bsnm{Chouldechova},~\bfnm{A}\binits{A.}} \AND
  \bauthor{\bsnm{Tibshirani},~\bfnm{R}\binits{R.}}
(\byear{2016}).
\btitle{{Sequential selection procedures and false discovery rate control}}.
\bjournal{Journal of the Royal Statistical Society: Series B (Statistical
  Methodology)}
\bvolume{78}
\bpages{423--444}.
\end{barticle}
\endbibitem

\bibitem[\protect\citeauthoryear{Hemerik, Solari and
  Goeman}{2019}]{Hemerik2019}
\begin{barticle}[author]
\bauthor{\bsnm{Hemerik},~\bfnm{J}\binits{J.}},
  \bauthor{\bsnm{Solari},~\bfnm{A}\binits{A.}} \AND
  \bauthor{\bsnm{Goeman},~\bfnm{J~J}\binits{J.~J.}}
(\byear{2019}).
\btitle{{Permutation-based simultaneous confidence bounds for the false
  discovery proportion}}.
\bjournal{Biometrika}
\bvolume{106}
\bpages{635--649}.
\bdoi{10.1093/biomet/asz021}
\end{barticle}
\endbibitem

\bibitem[\protect\citeauthoryear{Javanmard and Montanari}{2017}]{JM16}
\begin{barticle}[author]
\bauthor{\bsnm{Javanmard},~\bfnm{A}\binits{A.}} \AND
  \bauthor{\bsnm{Montanari},~\bfnm{A}\binits{A.}}
(\byear{2017}).
\btitle{{Online rules for control of false discovery rate and false discovery
  exceedance}}.
\bjournal{The Annals of Statistics}
\bvolume{46}.
\end{barticle}
\endbibitem

\bibitem[\protect\citeauthoryear{Katsevich and Sabatti}{2019}]{katsevich17mkf}
\begin{barticle}[author]
\bauthor{\bsnm{Katsevich},~\bfnm{Eugene}\binits{E.}} \AND
  \bauthor{\bsnm{Sabatti},~\bfnm{Chiara}\binits{C.}}
(\byear{2019}).
\btitle{{Multilayer knockoff filter: Controlled variable selection at multiple
  resolutions}}.
\bjournal{The Annals of Applied Statistics}
\bvolume{13}
\bpages{1--33}.
\end{barticle}
\endbibitem

\bibitem[\protect\citeauthoryear{Lehmann and
  Romano}{2005}]{lehmann2005generalizations}
\begin{barticle}[author]
\bauthor{\bsnm{Lehmann},~\bfnm{E~L}\binits{E.~L.}} \AND
  \bauthor{\bsnm{Romano},~\bfnm{J}\binits{J.}}
(\byear{2005}).
\btitle{{Generalizations of the familywise error rate}}.
\bjournal{The Annals of Statistics}
\bvolume{33}
\bpages{1138--1154}.
\end{barticle}
\endbibitem

\bibitem[\protect\citeauthoryear{Lei and Fithian}{2016}]{lei2016power}
\begin{binproceedings}[author]
\bauthor{\bsnm{Lei},~\bfnm{Lihua}\binits{L.}} \AND
  \bauthor{\bsnm{Fithian},~\bfnm{Will}\binits{W.}}
(\byear{2016}).
\btitle{{Power of ordered hypothesis testing}}.
In \bbooktitle{International Conference on Machine Learning}
\bpages{2924--2932}.
\end{binproceedings}
\endbibitem

\bibitem[\protect\citeauthoryear{Lei and Fithian}{2018}]{lei2016adapt}
\begin{barticle}[author]
\bauthor{\bsnm{Lei},~\bfnm{Lihua}\binits{L.}} \AND
  \bauthor{\bsnm{Fithian},~\bfnm{William}\binits{W.}}
(\byear{2018}).
\btitle{{AdaPT: an interactive procedure for multiple testing with side
  information}}.
\bjournal{Journal of the Royal Statistical Society: Series B (Statistical
  Methodology)}
\bvolume{80}
\bpages{649--679}.
\end{barticle}
\endbibitem

\bibitem[\protect\citeauthoryear{Lei, Ramdas and Fithian}{2017}]{lei17star}
\begin{barticle}[author]
\bauthor{\bsnm{Lei},~\bfnm{Lihua}\binits{L.}},
  \bauthor{\bsnm{Ramdas},~\bfnm{Aaditya}\binits{A.}} \AND
  \bauthor{\bsnm{Fithian},~\bfnm{William}\binits{W.}}
(\byear{2017}).
\btitle{{STAR: A general interactive framework for FDR control under structural
  constraints}}.
\bjournal{arXiv}.
\end{barticle}
\endbibitem

\bibitem[\protect\citeauthoryear{Li and Barber}{2017}]{LB17}
\begin{barticle}[author]
\bauthor{\bsnm{Li},~\bfnm{Ang}\binits{A.}} \AND
  \bauthor{\bsnm{Barber},~\bfnm{Rina~Foygel}\binits{R.~F.}}
(\byear{2017}).
\btitle{{Accumulation tests for FDR control in ordered hypothesis testing}}.
\bjournal{Journal of the American Statistical Association}
\bvolume{112}
\bpages{837--849}.
\end{barticle}
\endbibitem

\bibitem[\protect\citeauthoryear{McLean et~al.}{2010}]{MetB10}
\begin{barticle}[author]
\bauthor{\bsnm{McLean},~\bfnm{C~Y}\binits{C.~Y.}},
  \bauthor{\bsnm{Bristor},~\bfnm{D}\binits{D.}},
  \bauthor{\bsnm{Hiller},~\bfnm{M}\binits{M.}},
  \bauthor{\bsnm{Clarke},~\bfnm{S~L}\binits{S.~L.}},
  \bauthor{\bsnm{Schaar},~\bfnm{B~T}\binits{B.~T.}},
  \bauthor{\bsnm{Lowe},~\bfnm{C~B}\binits{C.~B.}},
  \bauthor{\bsnm{Wenger},~\bfnm{A~M}\binits{A.~M.}} \AND
  \bauthor{\bsnm{Bejerano},~\bfnm{G}\binits{G.}}
(\byear{2010}).
\btitle{{GREAT improves functional interpretation of cis-regulatory regions}}.
\bjournal{Nature Biotechnology}
\bvolume{28}
\bpages{495--501}.
\end{barticle}
\endbibitem

\bibitem[\protect\citeauthoryear{Meinshausen}{2006}]{meinshausen2006false}
\begin{barticle}[author]
\bauthor{\bsnm{Meinshausen},~\bfnm{N}\binits{N.}}
(\byear{2006}).
\btitle{{False discovery control for multiple tests of association under
  general dependence}}.
\bjournal{Scandinavian Journal of Statistics}
\bvolume{33}
\bpages{227--237}.
\end{barticle}
\endbibitem

\bibitem[\protect\citeauthoryear{Meinshausen and
  Rice}{2006}]{meinshausen2006estimating}
\begin{barticle}[author]
\bauthor{\bsnm{Meinshausen},~\bfnm{Nicolai}\binits{N.}} \AND
  \bauthor{\bsnm{Rice},~\bfnm{John}\binits{J.}}
(\byear{2006}).
\btitle{{Estimating the proportion of false null hypotheses among a large
  number of independently tested hypotheses}}.
\bjournal{The Annals of Statistics}
\bvolume{34}
\bpages{373--393}.
\end{barticle}
\endbibitem

\bibitem[\protect\citeauthoryear{Owen}{2005}]{owen2005variance}
\begin{barticle}[author]
\bauthor{\bsnm{Owen},~\bfnm{A~B}\binits{A.~B.}}
(\byear{2005}).
\btitle{{Variance of the number of false discoveries}}.
\bjournal{Journal of the Royal Statistical Society: Series B (Statistical
  Methodology)}
\bvolume{67}
\bpages{411--426}.
\end{barticle}
\endbibitem

\bibitem[\protect\citeauthoryear{Ramdas et~al.}{2017}]{RYWJ17}
\begin{binproceedings}[author]
\bauthor{\bsnm{Ramdas},~\bfnm{A}\binits{A.}},
  \bauthor{\bsnm{Yang},~\bfnm{F}\binits{F.}},
  \bauthor{\bsnm{Wainwright},~\bfnm{M~J}\binits{M.~J.}} \AND
  \bauthor{\bsnm{Jordan},~\bfnm{M~I}\binits{M.~I.}}
(\byear{2017}).
\btitle{{Online control of the false discovery rate with decaying memory}}.
In \bbooktitle{Advances In Neural Information Processing Systems}.
\end{binproceedings}
\endbibitem

\bibitem[\protect\citeauthoryear{Ramdas et~al.}{2018}]{SAFFRON}
\begin{binproceedings}[author]
\bauthor{\bsnm{Ramdas},~\bfnm{Aaditya}\binits{A.}},
  \bauthor{\bsnm{Zrnic},~\bfnm{Tijana}\binits{T.}},
  \bauthor{\bsnm{Wainwright},~\bfnm{Martin}\binits{M.}} \AND
  \bauthor{\bsnm{Jordan},~\bfnm{Michael}\binits{M.}}
(\byear{2018}).
\btitle{{SAFFRON: An adaptive algorithm for online control of the false
  discovery rate}}.
In \bbooktitle{Proceedings of the 35th International Conference on Machine
  Learning}
\bpages{4286--4294}.
\end{binproceedings}
\endbibitem

\bibitem[\protect\citeauthoryear{Robbins}{1954}]{robbins1954one}
\begin{binproceedings}[author]
\bauthor{\bsnm{Robbins},~\bfnm{H}\binits{H.}}
(\byear{1954}).
\btitle{{A one-sided confidence interval for an unknown distribution
  function}}.
In \bbooktitle{The Annals of Mathematical Statistics}
\bvolume{25}
\bpages{409}.
\end{binproceedings}
\endbibitem

\bibitem[\protect\citeauthoryear{Sesia, Sabatti and
  Cand{\`{e}}s}{2019}]{SetC17}
\begin{barticle}[author]
\bauthor{\bsnm{Sesia},~\bfnm{M.}\binits{M.}},
  \bauthor{\bsnm{Sabatti},~\bfnm{C.}\binits{C.}} \AND
  \bauthor{\bsnm{Cand{\`{e}}s},~\bfnm{E.~J.}\binits{E.~J.}}
(\byear{2019}).
\btitle{{Gene hunting with hidden Markov model knockoffs}}.
\bjournal{Biometrika}
\bvolume{106}
\bpages{1--18}.
\bdoi{10.1093/biomet/asy033}
\end{barticle}
\endbibitem

\bibitem[\protect\citeauthoryear{Sesia et~al.}{2019}]{SetS19}
\begin{barticle}[author]
\bauthor{\bsnm{Sesia},~\bfnm{M}\binits{M.}},
  \bauthor{\bsnm{Katsevich},~\bfnm{E}\binits{E.}},
  \bauthor{\bsnm{Bates},~\bfnm{S}\binits{S.}},
  \bauthor{\bsnm{Cand{\`{e}}s},~\bfnm{E}\binits{E.}} \AND
  \bauthor{\bsnm{Sabatti},~\bfnm{C}\binits{C.}}
(\byear{2019}).
\btitle{{Multi-resolution localization of causal variants across the genome}}.
\bjournal{bioRxiv}.
\end{barticle}
\endbibitem

\bibitem[\protect\citeauthoryear{Storey, Taylor and Siegmund}{2004}]{Storey04}
\begin{barticle}[author]
\bauthor{\bsnm{Storey},~\bfnm{J}\binits{J.}},
  \bauthor{\bsnm{Taylor},~\bfnm{J}\binits{J.}} \AND
  \bauthor{\bsnm{Siegmund},~\bfnm{D}\binits{D.}}
(\byear{2004}).
\btitle{{Strong control, conservative point estimation and simultaneous
  conservative consistency of false discovery rates: a unified approach}}.
\bjournal{Journal of the Royal Statistical Society: Series B (Statistical
  Methodology)}
\bvolume{66}
\bpages{187--205}.
\end{barticle}
\endbibitem

\bibitem[\protect\citeauthoryear{van~der Laan, Dudoit and
  Pollard}{2004}]{van2004multiple}
\begin{barticle}[author]
\bauthor{\bparticle{van~der} \bsnm{Laan},~\bfnm{M~J}\binits{M.~J.}},
  \bauthor{\bsnm{Dudoit},~\bfnm{S}\binits{S.}} \AND
  \bauthor{\bsnm{Pollard},~\bfnm{K~S}\binits{K.~S.}}
(\byear{2004}).
\btitle{{Multiple testing. Part III. Procedures for control of the generalized
  family-wise error rate and proportion of false positives}}.
\end{barticle}
\endbibitem

\bibitem[\protect\citeauthoryear{Ville}{1939}]{ville_etude_1939}
\begin{bbook}[author]
\bauthor{\bsnm{Ville},~\bfnm{J}\binits{J.}}
(\byear{1939}).
\btitle{{Etude critique de la notion de collectif}}.
\bpublisher{Gauthier-Villars}, \baddress{Paris}.
\end{bbook}
\endbibitem

\bibitem[\protect\citeauthoryear{Wellner}{1978}]{Wellner1978}
\begin{barticle}[author]
\bauthor{\bsnm{Wellner},~\bfnm{Jon}\binits{J.}}
(\byear{1978}).
\btitle{{Limit theorems for the ratio of the empirical distribution function to
  the true distribution function}}.
\bjournal{Zeitschrift f{\"{u}}r Wahrscheinlichkeitstheorie und verwandte
  Gebiete}
\bvolume{45}
\bpages{73--88}.
\end{barticle}
\endbibitem

\end{thebibliography}
\bibliographystyle{imsart-nameyear}

\newpage

\begin{center}
	\textbf{SUPPLEMENTARY MATERIALS}
\end{center}

\appendix

The method we employ to obtain all our FDP bounds is to investigate the properties of the stochastic process $\fdp(\cR_k)$. Below, we prove the proposed FDP bounds. We start with the sorted path construction, for which we must proceed differently since the hypotheses are ordered according to the p-values, while for all other settings, the hypothesis order is either pre-specified or is in some sense ``orthogonal'' to $\fdpbar$. 

\section{Proof of Theorem \ref{BH_THEOREM}} \label{app:BH}

Let $V_t \equiv \sum_{j \in \H_0} I(p_j \leq t)$, and let $V'_t$ be distributed as the unscaled empirical process of $n$ independent uniformly distributed random variables. Then, for any fixed $x > 1$ and $a = 1$ we have
\begin{equation}
\begin{split}
\PP{\sup_{1 \leq k \leq n} \frac{\fdp(\cR_k)}{\fdphat_a(\cR_k)} \geq x} &= \PP{\sup_{t \in [0,1]} \frac{V_t}{1+nt} \geq x} \\
&\leq \PP{\sup_{t \in [0,1]} \frac{V'_t}{1+nt} \geq x} \\
&= \PP{V'_t \geq x + xnt, \text{ for some } t \in [0,1]}.
\label{BH_stochastic_process}
\end{split}
\end{equation}
The second inequality holds because $V'_t$ is stochastically larger than $V_t$ because $n \geq |\mathcal H_0|$ and because the p-values are assumed to be stochastically larger than uniform. Hence, the quantity of interest is the probability that the stochastic process $V'_t$ hits the linear boundary $x+xnt$. This event is illustrated in Figure \ref{fig:poisson_process}.
\begin{figure}[h]
	\centering
	\includegraphics[width = 0.5\textwidth]{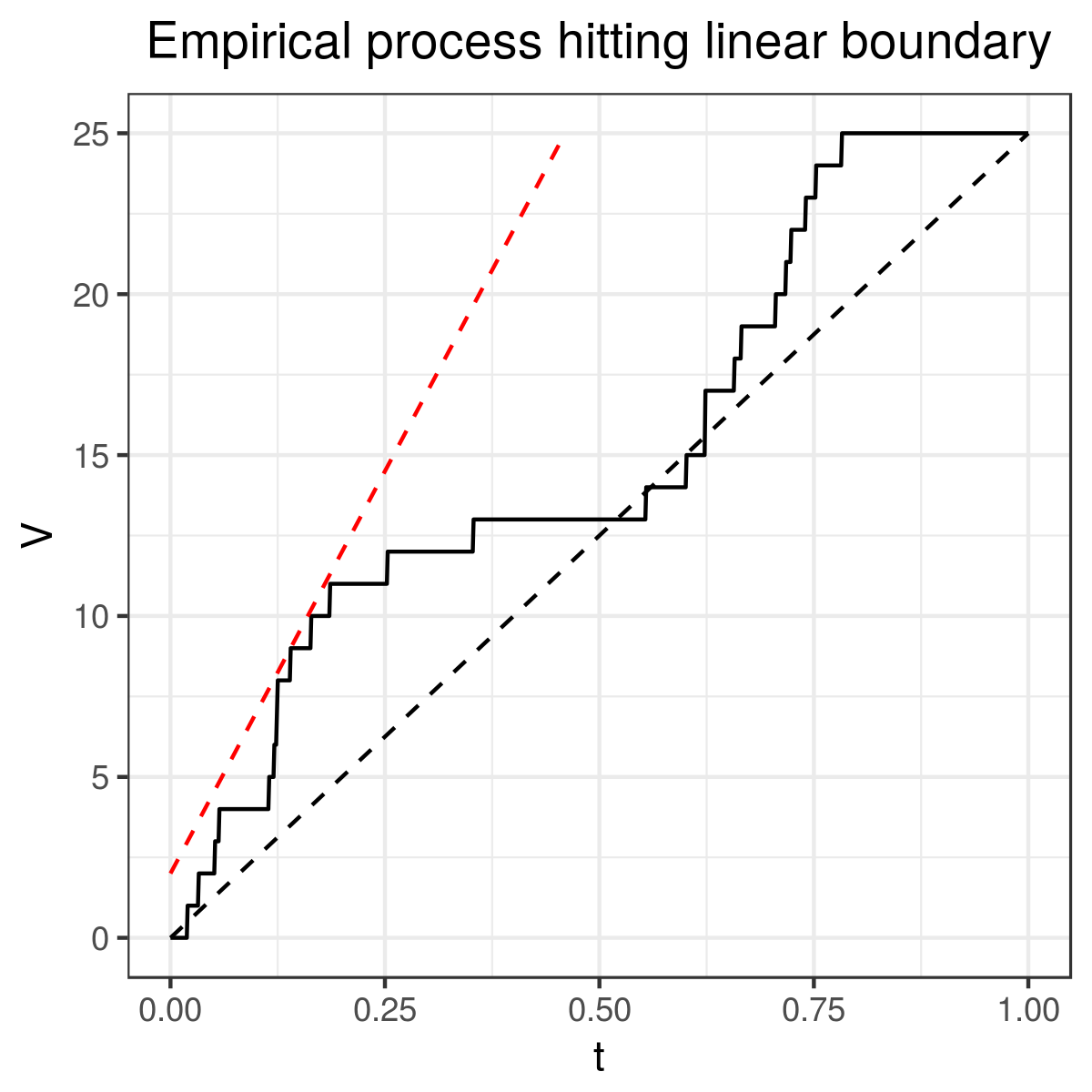}
	\caption{\small Stochastic process view of equation (\ref{BH_stochastic_process}), with the solid black line representing $V_t$, the dashed red line representing the linear boundary $x+xnt$, and the dashed black line being the diagonal that emphasizes our knowledge that $V(0)=0$ and $V(1)=n$. }
	\label{fig:poisson_process}
\end{figure}


However, it is somewhat difficult to obtain non-asymptotic bounds for probabilities that empirical processes like $V_t'$ hit certain boundaries. Instead, we claim that for $x \geq 1.5$, replacing $V_t'$ with a rate $n$ Poisson process $N_t$ only further increases the hitting probability:
\begin{lemma} \label{lemma:empirical_poisson}
	Let $V_t' \equiv \sum_{j = 1}^n I(p_j \leq t)$ for $p_j \overset{\text{i.i.d.}}{\sim} U[0,1]$. Let $N_t$ be a rate $n$ Poisson process. Then, for $x \geq 1.5$, 
	\begin{equation}
	\PP{V_t' \geq x + xnt, \text{ for some } t \in [0,1]} \leq \PP{N_t \geq x + xnt, \text{ for some } t \in [0,1]}.
	\label{empirical_poisson_inequality}
	\end{equation}
\end{lemma}
See Section \ref{sec:proof_lemma_1} for the proof of this lemma. To understand why the lemma holds, note that $V_t' \overset{d}= N_t|N_1 = n$; i.e. the distribution of the empirical process is the same as that of the Poisson process, conditioned on observing exactly $n$ events at time $t = 1$. Hence, Lemma \ref{lemma:empirical_poisson} states that for sufficiently large $x$, a Poisson process is less likely to hit the line $t \mapsto x + xnt$ if we know that it is equal to its mean at time $t = 1$. This makes intuitive sense because the process needs to be far above its mean to hit this line for large $x$, which is less likely to happen if it must be equal to its mean at time $t = 1$.

Now, we may use the martingale properties of Poisson processes to bound the probability a Poisson process hits a linear boundary:
\begin{lemma}[\cite{katsevich17mkf}]
	If $N_t$ is a rate $n$ Poisson process, then for any $x > 1$, we have
	\begin{equation*}
	\PP{N_t \geq x + xnt \text{ for some } t \in [0,1]} \leq \exp(-x\theta_x),
	\end{equation*}
	where $\theta_x$ is the unique positive root of the equation (\ref{theta_eq}) in the main text.
\end{lemma}
Therefore, 
\begin{equation*}
\PP{\sup_{1 \leq k \leq n}\frac{\fdp(\cR_k)}{\fdphat_a(\cR_k)} \geq } \leq \exp(-x\theta_x).
\end{equation*}		
To complete the derivation of the proof, suppose we choose $x$ such that $\alpha = \exp(-x\theta_x)$. Then, plugging this choice of $x$ into the definition of $\theta_x$ implies that
\begin{equation}
x = \frac{-\log \alpha}{\theta_x} = \frac{-\log \alpha}{\log(1 + \theta_x x)} = \frac{-\log \alpha}{\log(1 -\log \alpha)}.
\label{BH_multiplier_calculation}
\end{equation}
Hence, statement~\eqref{simultaneous_selective} holds for all $\alpha$ corresponding to $x \geq 1.5$, which translates to holding for all $\alpha \leq 0.31$.

\section{Proofs of Theorems \ref{ORDERED_THEOREM}, \ref{INTERACTIVE_THEOREM}, and \ref{ONLINE_THEOREM}} \label{app:234}

Next, we finish the proof of Lemma \ref{main_lemma} in the main text and derive Theorems \ref{ORDERED_THEOREM}, \ref{INTERACTIVE_THEOREM}, and \ref{ONLINE_THEOREM} as corollaries. 

\begin{proof}[Proof of Lemma \ref{main_lemma}]
	
	What remains to show is that $Z_k$ is a supermartingale. Note first of all that $Z_k$ is adapted to $\F_k$ by assumption (\ref{proof_filtration}). Hence, it suffices to show that 
	\begin{equation*}
	\EEst{\frac{Z_k}{Z_{k-1}}}{\F_{k-1}} = \EEst{\exp\left\{[\theta(I(p_{\pi(k)} \leq \alpha_k)-xh_k(p_{\pi(k)}))]I(\pi(k) \in H_0)\right\}}{\F_{k-1}} \leq 1.
	\end{equation*}
	Clearly, this inequality holds for any $k$ such that $\pi(k) \not \in \H_0$. For $k$ such that $\pi(k) \in \H_0$, we find that
	\begin{equation}\label{eq:pre-cases}
	\begin{split}
	&\EEst{\exp\left\{\theta(I(p_{\pi(k)} \leq \alpha_k)-xh_k(p_{\pi(k)}))\right\}}{\F_{k-1}} \\
	&\quad= \EEst{I(p_{\pi(k)} \leq \alpha_k)\exp\left\{\theta(1-xh_k(p_{\pi(k)}))\right\}}{\F_{k-1}} \\
	&\quad\quad+ \EEst{I(p_{\pi(k)} > \alpha_k)\exp\left\{-\theta xh_k(p_{\pi(k)}))\right\}}{\F_{k-1}}.
	\end{split}
	\end{equation}
	To show that the above quantity is at most one, we consider the four cases defined in the statement of the lemma.
	\paragraph{Case 1} 
	Since $\alpha_k = 1$, the second term of equation \eqref{eq:pre-cases} equals zero. Since $h_k = h$, $\pi(k)$ is fixed, and $p_{\pi(k)} \independent \F_{k-1}$, the first term simplifies to
	\begin{equation*}
	\begin{split}
	\EEst{\exp\left\{\theta_x(1-xh(p_{\pi(k)}))\right\}}{\F_{k-1}} 
	&~=~ \EE{\exp\left\{\theta_x(1-xh(p_{\pi(k)}))\right\}} \\
	&~\leq~ \exp(\theta_x)\Ep{U \sim U[0,1]}{\exp(-\theta_x x h(U))} \\
	&~=~ 1,
	\end{split}
	\end{equation*}
	The inequality holds because $p_{\pi(k)}$ is superuniformly distributed by assumption and $u \mapsto \exp(-\theta_x x h(u))$ is a nonincreasing function (since $h$ is nondecreasing by definition), and the last step holds because $\theta_x$ satisfies equation (\ref{theta_eq_acc}) by definition.
	
	\paragraph{Case 2} Again, the second term of equation (\ref{eq:pre-cases}) equals zero because $\alpha_k = 1$. We may bound the first term as:
	\begin{equation*}
	\begin{split}
	\EEst{\exp\left\{\theta_x(1-xh(p_{\pi(k)}))\right\}}{\F_{k-1}} 
	&~\leq~ \exp(\theta_x)\EEst{1 - \frac{1-\exp(-\theta_x x B)}{B}h(p_{\pi(k)})}{\F_{k-1}} \\
	&~\leq~ \exp(\theta_x)\left(1 - \frac{1-\exp(-\theta_x x B)}{B}\right) \\
	&~=~ 1.
	\end{split}
	\end{equation*}
	In the first line, we used the fact that $h_k(p_{\pi(k)}) = h(p_{\pi(k)}) \leq B$, and we bounded the convex function $z \mapsto \exp(-\theta_x x z)$ on $[0, B]$ with the line $z \mapsto 1-\frac{1-\exp(-\theta_x x B)}{B}z$. In the second step, we used the assumption $\EEst{h(p_{\pi(k)})}{\F_{k-1}} \geq \alpha_k = 1$, and in the third line we used the definition of $\theta_x$.
	
	\paragraph{Case 3} Because $h_k(p)I(p \leq \alpha_k) = 0$, the first term of equation (\ref{eq:pre-cases}) simplifies to
	\begin{equation*}
	\EEst{I(p_{\pi(k)} \leq \alpha_k)\exp\left\{\theta_x(1-xh_k(p_{\pi(k)}))\right\}}{\F_{k-1}} = \exp(\theta_x) \PPst{p_{\pi(k)} \leq \alpha_k}{\F_{k-1}}.
	\end{equation*}
	To bound the second term, we write
	\begin{equation*}
	\begin{split}
	&\EEst{I(p_{\pi(k)} > \alpha_k)\exp\left\{-\theta_x xh_k(p_{\pi(k)}))\right\}}{\F_{k-1}} \\ &\quad\leq \EEst{I(p_{\pi(k)} > \alpha_k)\left(1 - \frac{1-\exp(-\theta_x x B)}{B}h_k(p_{\pi(k)})\right)}{\F_{k-1}} \\
	&\quad= \PPst{p_{\pi(k)} > \alpha_k}{\F_{k-1}}  - \frac{1-\exp(-\theta_x x B)}{B}\EEst{h_k(p_{\pi(k)})I(p_{\pi(k)} > \alpha_k)}{\F_{k-1}} \\
	&\quad= \PPst{p_{\pi(k)} > \alpha_k}{\F_{k-1}}  - \frac{1-\exp(-\theta_x x B)}{B}\EEst{h_k(p_{\pi(k)})}{\F_{k-1}} \\
	&\quad\leq \PPst{p_{\pi(k)} > \alpha_k}{\F_{k-1}} - \frac{1-\exp(-\theta_x x B)}{B}\alpha_k.
	\end{split}
	\end{equation*}
	The inequality in the second line follows from the same convexity argument as in Case 2, the fourth line is a consequence of the fact that $h_k(p)I(p \leq \alpha_k) = 0$, and the last line follows from $\EEst{h_k(p_{\pi(k)})}{\F_{k-1}} \geq \alpha_k$.
	
	Combining the results of the previous two equations, we obtain that
	\begin{equation*}
	\begin{split}
	&\EEst{I(p_{\pi(k)} \leq \alpha_k)\exp\left\{\theta_x(1-xh_k(p_{\pi(k)}))\right\}}{\F_{k-1}} + \EEst{I(p_{\pi(k)} > \alpha_k)\exp\left\{-\theta_x xh_k(p_{\pi(k)}))\right\}}{\F_{k-1}} \\
	&\quad \leq  \exp(\theta_x) \PPst{p_{\pi(k)} \leq \alpha_k}{\F_{k-1}} + \PPst{p_{\pi(k)} > \alpha_k}{\F_{k-1}} - \frac{1-\exp(-\theta_x x B)}{B}\alpha_k \\
	&\quad =  (\exp(\theta_x)-1) \PPst{p_{\pi(k)} \leq \alpha_k}{\F_{k-1}} + 1 - \frac{1-\exp(-\theta_x x B)}{B}\alpha_k \\
	&\quad \leq  (\exp(\theta_x)-1) \alpha_k + 1 - \frac{1-\exp(-\theta_x x B)}{B}\alpha_k \\
	&\quad = \alpha_k\left(\exp(\theta_x)-1 - \frac{1-\exp(-\theta_x x B)}{B}\right) + 1 \\
	&\quad = 1.
	\end{split}
	\end{equation*}
	The inequality in the fourth line follows from the first part of assumption (\ref{p_h_requirement}), and the last equality follows from the definition of $\theta_x$.

	\paragraph{Case 4:}  Since $h_k(p) = \alpha_k$, equation \eqref{eq:pre-cases} simplifies to:
	\begin{equation*}
	\begin{split}
	&\exp(\theta_x(1-x\alpha_k))\PPst{p_{\pi(k)} \leq \alpha_k}{\F_{k-1}} + \exp(-\theta_x x \alpha_k)\PPst{p_{\pi(k)} > \alpha_k}{\F_{k-1}} \\
	&\quad = \exp(-\theta_x x\alpha_k)((\exp(\theta_x)-1)\PPst{p_{\pi(k)} \leq \alpha_k}{\F_{k-1}} + 1) \\
	&\quad \leq \exp(-\theta_x x\alpha_k)((\exp(\theta_x)-1)\alpha_k + 1).
	\end{split}
	\end{equation*}
	Noting that $\theta_x$ satisfies \eqref{theta_eq}, we see that the above expression can be bounded by one:
	\begin{equation*}
	\begin{split}
	\exp(-\theta_x x \alpha_k)\left(\theta_x x \alpha_k + 1\right) \leq 1,
	\end{split}
	\end{equation*}
	as desired, where the inequality follows because the function $z \mapsto e^{-z}(z+1)$ is decreasing and takes the value $1$ at $z = 0$.
\end{proof}

\begin{proof}[Proof of Theorem \ref{ORDERED_THEOREM}]
	In this case, the ordering $\pi$ is pre-specified, so we may assume without loss of generality that $\pi(j) = j$. First we note that the filtration
	\begin{equation}
	\F_k = \sigma(\H_0, \{p_j\}_{j \leq k, j \in \H_0})
	\label{ordered_filtration}
	\end{equation}
	satisfies the required condition (\ref{proof_filtration}). Next, we consider cases $\textnormal{preorder-acc}$ and $\textnormal{preorder-sel}$ separately.
	\begin{enumerate}
		\item $\textnormal{preorder-acc}$:
		For any $k \in \H_0$, since $\alpha_k = 1$, the first part of the requirement (\ref{p_h_requirement}) is trivially satisfied. To show the second part, we use the independence assumption in Theorem \ref{ORDERED_THEOREM} and the assumed superuniformity of null p-values to derive
		\begin{equation*}
		\EEst{h_k(p_k)}{\F_{k-1}} = \EE{h(p_k)} \geq \Ep{U \sim U[0,1]}{h(U)} = \int_0^1 h(u)du = 1 = \alpha_k.
		\end{equation*}
		The inequality follows because $p_k$ is superuniform and $h$ is nondecreasing.
		Hence, the assumptions of Lemma \ref{main_lemma}, case 1 are satisfied. Hence, the bound (\ref{general_FDP_bound}) holds with $\theta_x$ satisfying equation (\ref{theta_eq_acc}). In order to derive the constant~(\ref{acc_general_bound}), we simply note that for any $\alpha \in (0,1)$, we may choose $x$ such that $\exp(-a\theta_x x) = \alpha$, from which it follows that 
		\begin{equation*}
		x = \frac{-\log \alpha}{a\theta_x} = \frac{\log \alpha}{a\log \int_0^1 \exp(-\theta_x x h(u))du} = \frac{\log \alpha}{a\log \int_0^1 \alpha^{h(u)/a}du},
		\end{equation*}
		as desired.
		
		If the accumulation function is bounded by $B$, then the assumptions of Lemma \ref{main_lemma}, case 2 are also satisfied, so  the bound (\ref{general_FDP_bound}) holds with $\theta_x$ satisfying equation (\ref{theta_eq_B}). In order to derive the constant~(\ref{acc_seqstep}), we note that for any $\alpha \in (0,1)$, we may choose $x$ such that $\exp(-a\theta_x x) = \alpha$, from which it follows that 
		\begin{equation*}
		x = \frac{-\log \alpha}{a\theta_x} = \frac{\log \alpha}{a\log\left(1-\frac{1-\exp(-\theta_x x B)}{B}\right)} = \frac{\log \alpha}{a\log\left(1-\frac{1-\alpha^{B/a}}{B}\right)},
		\end{equation*}
		as desired. 
		
		\item $\textnormal{preorder-sel}$: We claim that the assumptions of Lemma \ref{main_lemma} are satisfied with $\F_{k}$ defined as in equation (\ref{ordered_filtration}). Indeed, fix $k \in \H_0$. Then, the assumed superuniformity of null p-values implies that
		\begin{equation*}
		\PPst{p_{k} \leq \alpha_k}{\F_{k-1}} = \PP{p_{k} \leq p_*} \leq p_* = \alpha_k
		\end{equation*}
		and
		\begin{equation*}
		\EEst{h_{k}(p_k)}{\F_{k-1}} = \EE{\frac{p_*}{1-\lambda} I(p > \lambda)} \geq p_* = \alpha_k.
		\end{equation*}
		Hence, by case 3 of Lemma \ref{main_lemma}, it follows that (\ref{general_FDP_bound}) holds for $\theta_x$ satisfying (\ref{theta_eq_AS}). For fixed $\alpha \in (0,1)$, define $x$ such that $\exp(-a\theta_x x) = \alpha$. Then, 
		\begin{equation*}
		x = \frac{-\log \alpha}{a\theta_x} = \frac{-\log \alpha}{a\log\left(1 + \frac{1-\exp(-\theta_x x B)}{B}\right)} = \frac{-\log \alpha}{a\log\left(1 + \frac{1-\alpha^{B/a}}{B}\right)},
		\end{equation*}
		which justifies the constant in equation~\eqref{c_AS}.
	\end{enumerate} 
\end{proof}

\begin{proof}[Proof of Theorem \ref{INTERACTIVE_THEOREM}]
	
	This proof is similar to that of Theorem \ref{ORDERED_THEOREM}, exception the filtrations must be somewhat more complicated to accommodate the interactivity of the path. For both the interact-acc and interact-sel path constructions, define the filtration
	\begin{equation*}
	\F_k = \sigma(\H_0, \{x_j, g(p_j)\}_{j \in [n]}, \{p_{\pi(j)}\}_{j \leq k}).
	\end{equation*}
	Since $\pi$ is predictable with respect to the filtration $\G_k$ (\ref{interactive_filtration}) by its definition and since $\F_k \supseteq G_k$, it follows that $\pi$ is also predictable with respect to $\F_k$. Hence, requirement (\ref{proof_filtration}) holds for $\F_k$ as defined above. Now, let $\pi(k) \in \H_0$. 
	
	For case $\text{interact-sel}$, the first part of condition (\ref{p_h_requirement}) holds because
	\begin{equation*}
	\begin{split}
	\PPst{p_{\pi(k)} \leq \alpha_k}{\F_{k-1}} &= \EEst{\PPst{p_{\pi(k)} \leq p_*}{g(p_{\pi(k)}), \pi(k)}}{\pi(k)} \\
	&\leq \text{Pr}_{U \sim U[0,1]}\{U \leq p_* \ | \ g(U)\} \\
	&= p_* \\
	&= \alpha_k.
	\end{split}
	\end{equation*}
	The first equality follows from the independence assumption on the p-values, the first inequality follows from the predictability of $\pi(k)$ and from the assumption that the null p-values have increasing densities (see the proof of Proposition 1 in \cite{lei17star}). Similarly,
	\begin{equation*}
	\begin{split}
	\EEst{h_k(p_{\pi(k)})}{\F_{k-1}} &= \EEst{\EEst{\frac{p_*}{1-\lambda}I(p_{\pi(k)} > \lambda)}{g(p_{\pi(k)}), \pi(k)}}{\pi(k)} \\
	&\geq \frac{p_*}{1-\lambda}\text{Pr}_{U \sim U[0,1]}\{U > \lambda \ | \ g(U)\} = p_* = \alpha_k.
	\end{split}
	\end{equation*}
	
	For case $\text{interact-acc}$, the first part of condition (\ref{p_h_requirement}) is trivially satisfied. To derive the second part, we write
	\begin{equation*}
	\begin{split}
	\EEst{h_k(p_{\pi(k)})}{\F_{k-1}} &= \EEst{\EEst{h(p_{\pi(k)})}{g(p_{\pi(k)}), \pi(k)}}{\pi(k)} \\
	&\geq \mathbb E_{U \sim U[0,1]}[h(U)|g(U)] \\
	&= 1 \\
	&= \alpha_k.
	\end{split}
	\end{equation*}
	The justification for this derivation is similar to that for ${\textnormal{interact-sel}}$, noting in addition that $ \mathbb E_{U \sim U[0,1]}[h(U)|g(U)] = 1$ by construction.
	
	Having established that the assumptions of Lemma \ref{main_lemma} are satisfied, the rest of the proof follows exactly as in the proof of Theorem \ref{ORDERED_THEOREM}.
\end{proof}

\begin{proof}[Proof of Theorem \ref{ONLINE_THEOREM}]
	
	Let $\G_k$ be the filtration defined in the statement of Theorem \ref{ONLINE_THEOREM}. Define
	\begin{equation}
	\F_k \equiv \sigma(\G_k, \H_0).
	\label{online_proof_filtration}
	\end{equation}
	Then, assumption (\ref{proof_filtration}) clearly holds for cases $\textnormal{online-simple}$ and $\textnormal{online-adaptive}$ by requirement~\eqref{online_predictable}. The first part of (\ref{p_h_requirement}) holds for both online procedures considered by assumption (\ref{online_assumption}). For case $\textnormal{online-simple}$, the second part holds because $h_k(p_k) = \alpha_k$. For case $\textnormal{online-adaptive}$, the second part holds because for $k \in \H_0$,
	\begin{equation*}
	\EEst{h_k(p_k)}{\F_{k-1}} = \EEst{\frac{\alpha_k}{1-\lambda_k}I(p_k > \lambda_k)}{\F_{k-1}} \geq \alpha_k,
	\end{equation*}
	where the last step follows from the predictability of $\alpha_k$ and $\lambda_k$ and the assumption (\ref{online_assumption}).
	
	Hence, $\textnormal{online-simple}$ and $\textnormal{online-adaptive}$ satisfy the assumptions of Lemma \ref{main_lemma}. The remainder of the proof corresponds exactly to analogous parts of the earlier proofs described for cases $\sort$ and $\textnormal{preorder-sel}$, respectively.
\end{proof}

\section{Lemmas supporting proof of Theorem \ref{BH_THEOREM}} \label{sec:proof_lemma_1}
\begin{proof}[Proof of Lemma \ref{lemma:empirical_poisson}]
	Let $\{N_t\}_{t \geq 0}$ be a Poisson process with rate $n$. It suffices to show that for $x \geq 1.5$, 
	\begin{equation}
	\mathbb P\left[\left.\sup_{t \in [0,1]}\frac{N_t}{1 + nt} \geq x\right| N_1 = n\right] \leq \mathbb P\left[\sup_{t \in [0,1]}\frac{N_t}{1 + nt} \geq x\right].
	\end{equation}
	Let us define
	\begin{equation*}
	\tau = \inf\left\{t: \frac{N_t}{1 + nt} \geq x\right\}.
	\end{equation*}
	We claim that it suffices to show that for $x \geq 1.5$,
	\begin{equation}
	\mathbb P[N_1 = n|\tau] \leq \mathbb P[N_1 = n] \quad \text{for all } \tau \geq 0.
	\label{sufficient}
	\end{equation}
	Indeed, it would then follow that
	\begin{equation*}
	\begin{split}
	\mathbb P\left[\left.\sup_{t \in [0,1]}\frac{N_t}{1 + nt} \geq x\right| N_1 = n\right] &= \mathbb P\left[\tau \leq 1| N_1 = n\right] \\
	&= \frac{\mathbb P[\tau \leq 1]\mathbb P[N_1 = n|\tau \leq 1]}{\mathbb P[N_1 = n]} \\
	&\leq \mathbb P[\tau \leq 1] \\
	&= \mathbb P\left[\sup_{t \in [0,1]}\frac{N_t}{1 + nt} \geq x\right].
	\end{split}
	\end{equation*}
	Note that for a given $x$, $x(1 + nt) > n$ for $t > \frac{1}{x} - \frac{1}{n}$. Hence, statement (\ref{sufficient}) is trivial for $\tau > \frac{1}{x} - \frac{1}{n}$, so we need only consider 
	\begin{equation}
	\tau \leq \frac{1}{x} - \frac{1}{n}.
	\label{tau_range}
	\end{equation}
	Define
	\begin{equation}
	f(\lambda, y) = e^{-\lambda}\frac{\lambda^y}{\Gamma(y+1)}.
	\label{def_f}
	\end{equation}
	This function $f: [0, \infty) \rightarrow \mathcal \R$ is equal to the probability mass function of the Poisson with parameter $\lambda$ when $y \in \{0, 1, 2, \dots\}$. We have 
	\begin{equation}
	\begin{split}
	\mathbb P[N_1 = n|\tau] &= \mathbb P[N_1 = n|N_\tau = \lceil x(1 + n\tau)\rceil] \\
	&= \mathbb P[N_{1-\tau} = n - \lceil x(1 + n\tau)\rceil] \\
	&= f(n(1-\tau), n - \lceil x(1 + n\tau)\rceil) \\
	&\leq f(n(1-\tau), n - x(1 + n\tau)) \\
	&\equiv g(\tau, x).
	\end{split}
	\label{def_g}
	\end{equation}
	The inequality follows by Lemma \ref{poisson_monotonicity} because $n - x(1+n\tau) \leq n(1-\tau) - x \leq n(1-\tau) - 1.5$. Define 
	\begin{equation}
	\begin{split}
	h(\tau, x) = \log g(\tau, x) &= \log \left(\exp(-n(1-\tau))\frac{(n(1-\tau))^{n - x(1+n\tau)}}{\Gamma(1 + n - x(1+n\tau))}\right) \\
	&= -n(1-\tau) + (n - x(1 + n\tau))\log(n(1-\tau)) - \log \Gamma(1 + n - x(1+n\tau)).
	\end{split}
	\end{equation}
	Note that
	\begin{equation}
	\begin{split}
	&\frac{\partial}{\partial \tau}h(\tau, x) \\
	&\quad= n -xn \log(n(1-\tau)) -\frac{n - x(1+n\tau)}{1 - \tau} + xn \psi(1 + n - x(1+n\tau)) \\
	&\quad\leq n -xn \log(n(1-\tau)) -\frac{n - x(1+n\tau)}{1 - \tau} + xn \left(\log(1.5 + n - x(1 + n\tau)) - \frac{1}{1.5 + n - x(1+n\tau)}\right) \\
	&\quad\equiv r(\tau, x).
	\end{split}
	\label{dh_bound}
	\end{equation}
	The inequality follows by Lemma \ref{lemma:digamma}.
	
	To prove the inequality (\ref{sufficient}), it suffices to show that 
	\begin{enumerate}
		\item[(a)] $g(0, x) \leq \mathbb P[N_1 = n]$ for all $x \geq 1.5$;
		\item[(b)] $r(\tau,1.5) \leq 0$ for each $\tau \leq \frac{1}{1.5} - \frac{1}{n}$;
		\item[(c)] $\frac{\partial}{\partial x}r(\tau,1.5) \leq 0$ for each $\tau \leq \frac{1}{1.5} - \frac{1}{n}$;
		\item[(d)] $r(\tau,x)$ is concave in $x$ for each $\tau$.
	\end{enumerate}
	Indeed, note that (c) and (d) imply that for each $\tau \leq \frac{1}{1.5} - \frac{1}{n}$, $r(\tau,x)$ is a decreasing function of $x$ for $x \geq 1.5$. Hence, for each $x \geq 1.5$ and each $\tau$, we have $r(\tau,x) \leq r(\tau,1.5) \leq 0$, where the last inequality follows from (b). Hence, by inequality (\ref{dh_bound}), $\frac{\partial h}{\partial \tau}(\tau, x) \leq r(\tau, x) \leq 0$ for each $x \geq 1.5$. This means that $h(\tau, x)$ is decreasing in $\tau$ for each $x \geq 1.5$, so $g(\tau, x)$ is decreasing in $\tau$ for each $x \geq 1.5$, from which it follows that $g(\tau, x) \leq g(0, x) \leq \mathbb P[N_1 = n]$, where the last inequality follows by (a).
	
	\paragraph{Proof of (a)} We have
	\begin{equation*}
	g(0, x) = f(n, n - x) \leq f(n, n - 1) = f(n,n) = \mathbb P[N_1 = n].
	\end{equation*}
	where the inequality follows by Lemma \ref{poisson_monotonicity} and the equality $f(n-1, n) = f(n,n)$ holds because
	\begin{equation*}
	\frac{f(n, n-1)}{f(n,n)} = \frac{\frac{n^{n-1}}{\Gamma(n)}}{\frac{n^n}{\Gamma(n+1)}} = \frac{\Gamma(n+1)}{n\Gamma(n)} = 1.
	\end{equation*}
	
	\paragraph{Proof of (b)} We have
	\begin{equation*}
	\begin{split}
	r(\tau, 1.5) &= n - 1.5n \log(n(1-\tau)) -\frac{n - 1.5(1+n\tau)}{1 - \tau} \\
	&\quad + 1.5n \left(\log(1.5 + n - 1.5(1 + n\tau)) - \frac{1}{1.5 + n - 1.5(1+n\tau)}\right) \\
	&= n - 1.5n \log(n(1-\tau)) -\frac{n - n\tau - 1.5 - 0.5 n\tau}{1 - \tau} + 1.5n \left(\log(n(1 - 1.5\tau)) - \frac{1}{n(1 - 1.5\tau)}\right) \\
	&= 1.5n \log\left(\frac{1 - 1.5\tau}{1-\tau}\right) +\frac{1.5 + 0.5 n\tau}{1 - \tau} -\frac{1.5}{1 - 1.5\tau} \\
	&= n\left(1.5\log\left(\frac{1 - 1.5\tau}{1-\tau}\right) + \frac{0.5\tau}{1 - \tau} \right) +1.5\left(\frac{1}{1-\tau} - \frac{1}{1 - 1.5\tau}\right) \\
	&\leq n\left(1.5\log\left(\frac{1 - 1.5\tau}{1-\tau}\right) + \frac{0.5\tau}{1 - \tau} \right)  \\
	&= n\left(1.5\log\left(1 - \frac{0.5\tau}{1 - \tau}\right) + \frac{0.5\tau}{1 - \tau}\right) \\
	&\leq n\left(-1.5 \frac{0.5\tau}{1 - \tau} + \frac{0.5\tau}{1 - \tau}\right) \\
	&= -n\frac{0.25\tau}{1 - \tau} \\
	&\leq 0.
	\end{split}
	\end{equation*}
	\paragraph{Proof of (c)}
	We have
	\begin{equation*}
	\begin{split}
	\frac{\partial r}{\partial x}(\tau, x) &= -n\log(n(1-\tau)) + \frac{1+n\tau}{1-\tau} + n\log(1.5 + n - x(1 + n\tau)) \\
	&\qquad -\frac{(1+n\tau)xn}{1.5 + n - x(1+n\tau)} -n\frac{1.5 + n - x(1+n\tau) + x(1+n\tau)}{(1.5 + n - x(1+n\tau))^2} \\
	&= -n\log(n(1-\tau)) + \frac{1+n\tau}{1-\tau} + n\log(1.5 + n - x(1 + n\tau)) \\
	&\qquad -\frac{(1+n\tau)xn}{1.5 + n - x(1+n\tau)} -n\frac{1.5 + n }{(1.5 + n - x(1+n\tau))^2}.
	\end{split}
	\end{equation*}
	Plugging in $x = 1.5$, we have
	\begin{equation*}
	\begin{split}
	\frac{\partial r}{\partial x}(\tau, 1.5) &= -n\log(n(1-\tau)) + \frac{1+n\tau}{1-\tau} + n\log(n(1-1.5\tau)) -\frac{1.5(1+n\tau)n}{n(1-1.5\tau)} -n\frac{1.5 + n }{(n(1-1.5\tau))^2} \\
	&= n\log\left(\frac{1-1.5\tau}{1-\tau}\right) + \frac{1}{1-\tau} + n\frac{\tau}{1-\tau} - \frac{1.5}{1 - 1.5\tau}(1+n\tau) - \frac{1}{n}\frac{1.5}{(1-1.5\tau)^2} - \frac{1.5}{(1-1.5\tau)^2} \\
	&\leq n\left(\log\left(\frac{1-1.5\tau}{1-\tau}\right) + \frac{\tau}{1-\tau} - \frac{1.5\tau}{1-1.5\tau}\right) + \left(\frac{1}{1-\tau} - \frac{1.5}{1 -1.5\tau}\right) \\
	&\leq n\left(\log\left(\frac{1-1.5\tau}{1-\tau}\right) + \frac{\tau}{1-\tau} - \frac{1.5\tau}{1-1.5\tau}\right) \\
	&\leq n\left(\frac{-0.5\tau}{1-\tau} + \frac{\tau}{1-\tau} - \frac{1.5\tau}{1-1.5\tau}\right). \\
	&= n\left(\frac{0.5\tau}{1-\tau} - \frac{1.5\tau}{1-1.5\tau}\right). \\
	&\leq n \frac{\tau(0.75\tau - 1)}{(1-\tau)(1 - 1.5\tau)} \\
	&\leq 0,
	\end{split}
	\end{equation*}
	where the last inequality follows because $\tau \leq \frac{1}{1.5} - \frac{1}{n} \leq 4/3$.
	
	\paragraph{Proof of (d)}
	
	Modulo terms linear in $x$ and the scaling factor $n$, $r(\tau, x)$ is equal to 
	\begin{equation*}
	x\log\left(1.5 + n - x(1+n\tau)\right) - \frac{x}{1.5 + n - x(1+n\tau)}.
	\end{equation*}
	We claim that the first term is concave in $x$ and the second term is convex, from which it will follow that their difference is concave. By linear transformations, the concavity of the first term will follow from the concavity of $x\log(1-x)$, which follows because its first derivative $\log(1-x) - \frac{x}{1-x}$ is decreasing in $x$. Again by linear transformations, the convexity of the second term will follow from the convexity of $\frac{x}{1-x} = -1 + \frac{1}{1-x}$ on $x < 1$, which is clear.
\end{proof}

\begin{lemma} \label{lemma:digamma}
	Let $\psi(x) = \Gamma'(x)/\Gamma(x)$ be the digamma function. Then, $\psi$ is increasing, and for $x \geq 1$, 
	\begin{equation}
	\psi(x) \leq \log(x) - \frac{1}{2x} \leq \log(x + 0.5) - \frac{1}{x + 0.5} \leq \log x.
	\end{equation}
\end{lemma}
\begin{proof}
	The fact that $\psi$ is increasing is well-known. The first inequality follows directly from Theorem 3.1 of \cite{anderson1995inequalities}. To prove the second inequality, write 
	\begin{equation*}
	\log(x+0.5) - \log(x) = \log\left(1 + \frac{0.5}{x}\right) \geq \frac{0.5}{x}-\frac{1}{2}\left(\frac{0.5}{x}\right)^2,
	\end{equation*}
	and the conclusion follows because 
	\begin{equation*}
	\frac{0.5}{x}-\frac{1}{2}\left(\frac{0.5}{x}\right)^2 \geq \frac{1}{x + 0.5} - \frac{1}{2x} \  \Longleftrightarrow \ \frac{1}{2x} - \frac{1}{8x^2} \geq \frac{x - 0.5}{2x(x+0.5)} \ \Longleftrightarrow \ x \geq 1/6. 
	\end{equation*}
	To prove the third inequality, write
	\begin{equation*}
	\log(x+0.5) - \log(x) = \log\left(1 + \frac{0.5}{x}\right) \leq \frac{0.5}{x},
	\end{equation*}
	and the conclusion follows because
	\begin{equation*}
	\frac{0.5}{x} \leq \frac{1}{x+0.5} \ \Longleftrightarrow x \geq 0.5.
	\end{equation*}
\end{proof}

\begin{lemma} \label{poisson_monotonicity}
	Let $f$ be as defined in equation (\ref{def_f}). Then, 
	\begin{equation*}
	\frac{\partial}{\partial y}f(\lambda, y) \geq 0 \quad \textnormal{for all } y \leq \lambda - 1.
	\end{equation*}
\end{lemma}
\begin{proof}
	To prove this, it suffices to show that for $y \leq \lambda - 1$, 
	\begin{equation*}
	0 \leq \frac{\partial}{\partial y}\log f(\lambda, y) = \frac{\partial}{\partial y}\left(-\lambda + y \log \lambda - \log \Gamma(y + 1)\right) = \log \lambda - \psi(y+1).
	\end{equation*}
	Indeed, by Lemma \ref{lemma:digamma}, for $y \leq \lambda - 1$ we have $\psi(y+1) \leq \psi(\lambda) \leq \log \lambda$.
\end{proof}

\end{document}